\documentclass[a4paper,8pt,oneside]{article}
\usepackage[T1]{fontenc}
\usepackage[utf8]{inputenc}
\usepackage[english]{babel}
\usepackage{amsmath}
\usepackage{amsthm}
\usepackage{amssymb}
\usepackage{makeidx}
\usepackage{graphicx}
\usepackage{enumitem}
\usepackage[hidelinks]{hyperref}
\usepackage[left=4cm,right=4cm,top=3cm,bottom=3cm]{geometry}
\usepackage[usenames,dvipsnames,svgnames,table]{xcolor}
\newtheorem{thm}{Theorem}[section]

\newtheorem{defn}[thm]{Definition}
\newtheorem{rem}[thm]{Remark}	
\newtheorem{prop}[thm]{Proposition}
\newtheorem{cor}[thm]{Corollary}
\newtheorem{note}[thm]{Notation}

\usepackage{tikz}
\newcommand{\dashedtop}{
\begin{tikzpicture}[scale=0.5]
\draw [dashed] (-0.5,0) -- (0.5,0);
\draw [dashed] (0,0) -- (0,-0.5);
\end{tikzpicture}}
\usepackage{pstricks}
\usetikzlibrary{decorations.pathmorphing, patterns,shapes}
\usetikzlibrary{arrows,shapes,positioning}
\usetikzlibrary{decorations.markings}
\tikzstyle arrowstyle=[scale=1]
\tikzstyle directed=[postaction={decorate,decoration={markings,
    mark=at position .65 with {\arrow[arrowstyle]{stealth}}}}]
\tikzstyle reverse directed=[postaction={decorate,decoration={markings,
    mark=at position .65 with {\arrowreversed[arrowstyle]{stealth};}}}]
\usepackage{lipsum}
\usepackage{url}

\begin{document}


\title{Real algebraic curves  of bidegree (5,5) on the quadric ellipsoid}
\date{}
\author{Matilde Manzaroli}
\maketitle
\begin{abstract}
We complete the topological classification of non-separating (resp. separating) non-singular real algebraic $(M-i)$-curves of bidegree $(5, 5)$ on the quadric ellipsoid. In particular, we show that previously known restrictions form a complete system for this bidegree. Therefore, the main part of the paper concerns the construction of real algebraic curves. Our strategy is to reduce the problem of construction of curves on the quadric ellipsoid to the construction of curves on the second Hirzebruch surface by degenerating the quadric ellipsoid to the quadratic cone.  
Next, we combine  different classical construction methods on toric surfaces, such as dessins d'enfants and Viro's patchworking method. 
\end{abstract}
\tableofcontents
\section{Introduction}
\label{sec: introintro}
A real algebraic variety is a real algebraic compact complex variety $X$ equipped with an anti-holomorphic involution $\sigma: X \rightarrow X$, called real structure on $X$. The real part $\mathbb{R}X$ of $X$ is the set of points fixed by $\sigma$. A subject of interest about real algebraic varieties is the study of the topology of their real part and it is related to the classical Hilbert's $16$-th problem whose first part is about the classification of the oval arrangements of non-singular real algebraic plane curves. In this paper, we are interested in the study of the topology of real algebraic curves in the quadric ellipsoid. We give some definitions and notations in Section \ref{subsec: quadric}, we state the main results in Section \ref{subsec: main_results} and explain the structure of the paper in Section \ref{subsec: structure}.\\
First of all, let us present some general definitions and known results about real algebraic curves. \\
Let $(A,\sigma)$ be a compact non-singular real algebraic curve and let $l$ denote the number of connected components of the set of real points of $A$. The following inequality has been proven by Harnack for real algebraic plane curves and, after, generalized by Klein. 
\begin{prop}[\cite{Harn76}, \cite{Klei22}]
\label{prop: Harnack-Klein}
The number $l$ is bounded by $g+1$, where $g$ is the genus of $A$. 
\end{prop}
\begin{defn}
If $l=g+1$, we say that $A$ is an \textit{$M$-curve} or a \textit{maximal} curve. Otherwise, for $0\leq l \leq g$ , we say that $A$ is an \textit{$(M-i)$-curve} with $i=g+1-l$. 
\end{defn}
\begin{defn}[\cite{Klei22}]
If $A \setminus \mathbb{R}A$ is connected, we say that $A$ is of \textit{type II} or \textit{non-separating}, otherwise of \textit{type I} or \textit{separating}. 
\end{defn}
Looking at the real part of the curve and its position with respect to its complexification gives us information about $l$ and viceversa. For example, we know that if $A$ is maximal, then $A$ is of type I. Or, if $A$ is of type I then $l$ has the parity of $g+1$. 
\begin{defn}[\cite{Rokh72}]
If $A$ is of type I, the two  halves of $A\setminus \mathbb{R}A$ induce two opposite orientations on $\mathbb{R}A$ called \textit{complex orientations} of the curve.
\end{defn}
Variations of the $16$-th Hilbert problem propose to study topological classifications of real algebraic curves up to different equivalence relations (up to homeomorphism, resp. isotopy, resp. rigid isotopy). Let $(X, \sigma)$ be a real algebraic surface and let $A$ and $B$ be real algebraic curves realizing a certain class $\alpha$ in $H_2(X; \mathbb{Z})$.  
\begin{defn}
\label{defn: homeo}
We say that the pairs $(\mathbb{R}X, \mathbb{R}A)$ and $(\mathbb{R}X, \mathbb{R}B)$ are homeomorphic if there exists a homeomorphism $f: \mathbb{R}X \rightarrow \mathbb{R}X$ such that $f(\mathbb{R}A)=\mathbb{R}B$.
\end{defn}
\begin{defn}
\label{defn: isotopy}
We say that the pairs $(\mathbb{R}X, \mathbb{R}A)$ and $(\mathbb{R}X, \mathbb{R}B)$ are isotopic if the pairs $(\mathbb{R}X, \mathbb{R}A)$ and $(\mathbb{R}X, \mathbb{R}B)$ are homeomorphic via an homeomorphism $f$ and there exists a $1$-parameter family of homeomorphisms $\phi_t: \mathbb{R}X \rightarrow \mathbb{R}X$, $t \in [0,1]$, such that $\phi_0=id$ and $\phi_1=f$. 
\end{defn}
\begin{defn}
\label{defn: rigid_isotopy}
Let $A$ and $B$ be non-singular. We say that the pairs $(\mathbb{R}X, \mathbb{R}A)$ and $(\mathbb{R}X, \mathbb{R}B)$ are rigid isotopic if the pairs $(\mathbb{R}X, \mathbb{R}A)$ and $(\mathbb{R}X, \mathbb{R}B)$ are isotopic via an homeomorphism $f$ and a $1$-parameter family of homeomorphisms $\phi_t$ such that $\phi_t(\mathbb{R}A)$ is the set of real points of a nonsingular real algebraic curve of class $\alpha$ in $H_2(X; \mathbb{Z})$, for all $t \in [0,1]$. 
\end{defn}
\subsection{Quadric ellipsoid}
\label{subsec: quadric}
Let $X$ be $\mathbb{C} P^1 \times \mathbb{C} P^1$, equipped with the anti-holomorphic involution 
$$
\begin{array}{cccc}
\sigma: & X & \longrightarrow & X \\
 & \tiny{(x,y)} & \longmapsto & \tiny{(\overline{y},\overline{x})},
\end{array}
$$ 
where $x=[x_0:x_1]$ and $y=[y_0:y_1]$ are in $\mathbb{C}P^1$ and $\overline{x}=[\overline{x_0}:\overline{x_1}]$ and $\overline{y}=[\overline{y_0}:\overline{y_1}]$ are respectively the images of $x$ and $y$ via the standard complex conjugation on $\mathbb{C}P^1$. The real part of $X$ is homeomorphic to $S^2$. It is well known that $X$ is isomorphic to the quadric ellipsoid in $\mathbb{C}P^3$, in the category of real algebraic varieties.  
A non-singular real algebraic curve $A$ on $X$ is defined by a bi-homogeneous polynomial of bidegree $(d,d)$ $$P(x_0,x_1,y_0,y_1)=\sum_{i,j=1}^{d,d} a_{i,j}x^i_1x^{d-i}_0y^j_1y^{d-j}_0$$ where $d$ is a positive integer and the coefficients satisfy $a_{i,j}=\overline{a_{j,i}}$. Let $\mathbb{R}A$ be the set of real points of $A$. 
The connected components of $\mathbb{R}A$ are called \textit{ovals}. We are interested in the classification of the oval arrangements of non-singular real algebraic curves on $X$. Thanks to the adjunction formula, for a non-singular real algebraic curve $A$ of bidegree $(d,d)$ on $X$ we have $g=(d-1)^2$. It follows that the number of ovals of $\mathbb{R}A$ is bounded by $(d-1)^2+1$ (Proposition \ref{prop: Harnack-Klein}).\\
The rigid isotopy classification of non-singular real algebraic curves of bidegree $(d,d)$ is known for $d < 5$ (\cite{GudShu80}). For a more general idea on topological classifications of real curves on quadric surfaces see \cite{Zvon91}, \cite{NikSai05}, \cite{NikSai07},\cite{NikuSait05}, \cite{DegZvo99}, \cite{Niku85}, \cite{Mikh94} and \cite{DegKha00}. Moreover, Mikhalkin in \cite{Mikh94} has given a partial classification, up to homeomorphism, of real algebraic $M$-curves of bidegree $(5,5)$ on the quadric ellipsoid. In this paper, for all $0 \leq l \leq 17$, we give the complete classification, up to homeomorphism, of the topological types realized by the pair $(\mathbb{R}X, \mathbb{R}A)$, where $A$ is a type II or a type I non-singular real algebraic $(M-i)$-curve of bidegree $(5,5)$ on $X$, with $i=g+1-l$ (Theorems \ref{thm: maximal_thm}, \ref{thm: m-1_curves_thm}, \ref{thm: type_I,II_m-2_thm}, \ref{thm: type_I_II_thm}).\\
Before stating the main results of this paper, let us introduce some notation. 
Given a collection $\sqcup_{i=1,..l}B_i$ of $l$ disjoint circles embedded in $S^2$, we encode the topological pair $(S^2, \sqcup_{i=1,..l}B_i)$ as follows. Let $p$ be any point in $S^2 \setminus \sqcup_{i=1,..l}B_i$. Each oval $B_i$ bounds two non-homeomorphic parts in $S^2 \setminus \{p\}$. Let us call the \textit{interior} the part homeomorphic to a disc and the \textit{exterior} the other one. For each pair of ovals, if one is in the interior of the other we speak about an \textit{injective pair}, otherwise of a \textit{non-injective} one. We shall adopt the following notation to encode a given topological pair $(S^2\setminus \{p\},\sqcup_{i=1,..l}B_i)$. An empty union of ovals is denoted by $0$. We say that a union of $l$ ovals realizes $l$ if there are no injective pairs. The symbol $\langle \mathcal{S}  \rangle$ denotes the disjoint union of a collection of ovals realizing $\mathcal{S}$, and an oval forming an injective pair with each oval of the collection. Finally, the disjoint union of any two collections of ovals, realizing in $S^2 \setminus \{p\}$ respectively $\mathcal{S}'$ and $\mathcal{S}''$, is denoted by $\mathcal{S}'\quad \sqcup \quad \mathcal{S}''$ if none of the ovals of one collection forms an injective pair with the ovals of the other one.\\
We say that the pair $(S^2, \sqcup_{i=1,..,l} B_i)$ realizes $\mathcal{S}$ if there exists a point $p\in S^2 \setminus \sqcup_{i=1,..,l}B_i$ such that $(S^2 \setminus \{p\},  \sqcup_{i=1,..,l}B_i)$ realizes $\mathcal{S}$. As example, we have depicted in $a)$ of Fig. \ref{fig: example_real_scheme_vha2} an arrangement of $11$ ovals in $S^2$ projected on a plane from some point $p \in S^2$. The pair $(S^2, \bigsqcup_{i=1,..,8} B_i)$ realizes $1 \quad \sqcup \quad \langle 2\rangle \quad \sqcup \quad \langle 1\rangle \quad \sqcup \quad \langle \langle 3\rangle \rangle$.
\begin{figure}[!h]
\begin{picture}(100,50)
\put(-7,-19){\includegraphics[width=1.00\textwidth]{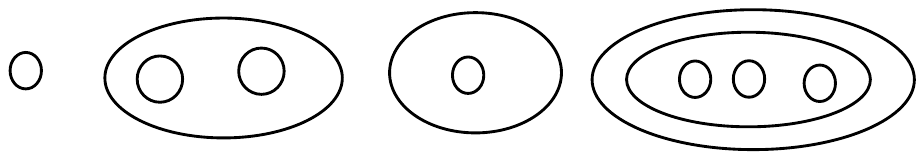}} 
\end{picture}
\caption{Example of an arrangement of embedded circles in $S^2\setminus \{p\}$.}
\label{fig: example_real_scheme_vha2}
\end{figure}
\begin{defn}
\label{defn: real scheme}
The topological type $\mathcal{S}$ realized by a pair $(\mathbb{R}X, \mathcal{B})$, where $\mathcal{B}$ is a collection of disjoint circles embedded in $\mathbb{R}X$, is called real scheme. Let $A$ be a real algebraic curve of bidegree $(d,d)$ on $X$, we say that $A$ has real scheme $\mathcal{S}$ if the pair $(\mathbb{R}X, \mathbb{R}A)$ realizes $\mathcal{S}$.
\end{defn}
\subsection{Main results}
\label{subsec: main_results}
We show that previous known results on the topological obstructions for bidegree $(d,d)$ real algebraic curves of type II or of type I form a complete system of restrictions for bidegree $(5,5)$. Indeed, all real schemes listed in Theorems \ref{thm: maximal_thm}, \ref{thm: m-1_curves_thm}, \ref{thm: type_I,II_m-2_thm}, \ref{thm: type_I_II_thm} are those non-prohibited by Propositions \ref{prop: Harnack-Klein}, \ref{prop: bezout}, \ref{prop: orientation_formula} and Theorem \ref{thm: mikha_congruences}; moreover, such real schemes are all realizable by type II or type I real algebraic $(M-i)$-curves of bidegree $(5,5)$ in $X$. The main results of this paper are Theorems \ref{thm: maximal_thm}, \ref{thm: m-1_curves_thm}, \ref{thm: type_I,II_m-2_thm} and \ref{thm: type_I_II_thm}.
\begin{thm}[$M$-curves]
\label{thm: maximal_thm}
Let $A$ be a non-singular real algebraic $M$-curve of bidegree $(5,5)$ on $X$. Then the pair $(\mathbb{R}X, \mathbb{R}A)$ realizes one of the following real schemes:
$$\alpha \quad \sqcup \quad  \langle \beta \rangle \quad \sqcup \quad  \langle \gamma  \rangle, \text{  } \alpha \equiv 1 \pmod{4}, \text{  } \text{with} \text{  }\alpha + \beta + \gamma=15.$$ 
Moreover, all such real schemes are realizable by non-singular real algebraic $M$-curves of bidegree $(5,5)$.
\end{thm}
\begin{proof}
The real schemes listed above are those non-prohibited by Propositions \ref{prop: Harnack-Klein}, \ref{prop: bezout} and $(1)$ of Theorem \ref{thm: mikha_congruences}. In Propositions \ref{prop: big_first_gluing}, \ref{prop: big_first_gluing_2} and \ref{prop: big_first_gluing_3}, we construct non-singular real algebraic $M$-curves of bidegree $(5,5)$ on $X$ realizing all real schemes listed above.
\end{proof}
\begin{thm}[($M-1$) curves]
\label{thm: m-1_curves_thm}
Let $A$ be a non-singular real algebraic ($M-1$)-curve of bidegree $(5,5)$ on $X$. Then the pair $(\mathbb{R}X, \mathbb{R}A)$ realizes one of the following real schemes:
$$\alpha \quad \sqcup \quad \langle\beta \rangle \quad \sqcup \quad  \langle \gamma  \rangle, \text{ } \alpha \equiv 0 \text{ } \text{or} \text{ 1} \pmod{4}, \text{ } \text{with} \text{ } \alpha + \beta + \gamma=14.$$
Moreover, all such real schemes are realizable by non-singular real algebraic ($M-1$)-curves of bidegree $(5,5)$.
\end{thm}
\begin{proof}
The real schemes listed above are those non-prohibited by Propositions \ref{prop: Harnack-Klein}, \ref{prop: bezout} and $(2)$ of Theorem \ref{thm: mikha_congruences}. In Propositions \ref{prop: big_first_gluing}, \ref{prop: big_first_gluing_2} and \ref{prop: big_first_gluing_3}, we construct non-singular real algebraic $(M-1)$-curves of bidegree $(5,5)$ on $X$ realizing all real schemes listed above.
\end{proof}
\begin{thm}[($M-2$)-curves]
\label{thm: type_I,II_m-2_thm}
Let $A$ be a non-singular real algebraic ($M-2$)-curve of bidegree $(5,5)$ on $X$. If $A$ is of type I, then the pair $(\mathbb{R}X, \mathbb{R}A)$ realizes one of the following real schemes:
$$(1) \text{  } \alpha \quad \sqcup \quad \langle \beta \rangle \quad \sqcup \quad  \langle \gamma  \rangle, \text{  } \alpha \equiv 0 \pmod{2}, \text{  } \text{with} \text{ } \alpha + \beta + \gamma=13;$$
if $A$ is of type II, then one of the following ones:
$$(2) \text{  } \alpha \quad \sqcup \quad \langle\beta  \rangle \quad \sqcup \quad \langle \gamma  \rangle, \quad \alpha \not\equiv 2\pmod{4},\quad \text{with} \quad \alpha + \beta + \gamma=13.$$
Moreover, all the real schemes in $(1)$ and $(2)$ are realizable by non-singular real algebraic ($M-2$)-curves of bidegree $(5,5)$ respectively of type I and II.
\end{thm}
\begin{proof}
The real schemes listed above are those non-prohibited by Propositions \ref{prop: Harnack-Klein}, \ref{prop: bezout} and $(3)$ of Theorem \ref{thm: mikha_congruences}. In Propositions \ref{prop: big_first_gluing}, and \ref{prop: big_first_gluing_2}, we construct non-singular real algebraic $(M-2)$-curves of bidegree $(5,5)$ on $X$ realizing all real schemes listed above.
\end{proof}
\begin{thm}[Type I and II  curves]
\label{thm: type_I_II_thm}
Let $A$ be a non-singular real algebraic ($M-i$)-curve of bidegree $(5,5)$ on $X$, where  $3 \leq i \leq 17$. If $A$ is of type II, then the pair $(\mathbb{R}X, \mathbb{R}A)$ realizes one of the following real schemes:
$$(1) \text{ } 0 \text{ } \text{and} \text{ } 1,$$
$$(2) \text{ } \alpha \quad \sqcup \quad  \langle \beta  \rangle \quad \sqcup \quad  \langle \gamma  \rangle, \quad \text{with} \quad \alpha + \beta + \gamma=17-(i-2);$$
if $A$ is of type I, then $i=4,6,8,10,12$ and the pair $(\mathbb{R}X, \mathbb{R}A)$ realizes one of the following real schemes:
$$(3) \text{ } \langle \langle \langle \langle 1 \rangle \rangle \rangle \rangle,$$
$$(4) \text{ } \alpha \quad \sqcup \quad  \langle \beta  \rangle \quad \sqcup \quad \langle \gamma  \rangle, \quad \text{with} \quad  \alpha \equiv 0 \pmod{2} \text{  } \text{when} \text{  } \alpha + \beta + \gamma=5,9,$$
$$\text{  } \text{and with} \text{ }  \alpha \equiv 1 \pmod{2} \text{  } \text{when} \text{  } \alpha + \beta + \gamma=7,11.$$
Moreover, the real schemes in $(1),(2)$ and $(3),(4)$ are realizable by non-singular real algebraic curves of bidegree $(5,5)$ respectively of type II and I.
\end{thm}
\begin{proof}
The real schemes listed above are those non-prohibited by Propositions \ref{prop: Harnack-Klein}, \ref{prop: bezout}, \ref{prop: orientation_formula} and $(4)$ of Theorem \ref{thm: mikha_congruences}. In Propositions \ref{prop: big_first_gluing}, \ref{prop: big_first_gluing_2} and \ref{prop: big_first_gluing_3}, we construct non-singular real algebraic $(M-i)$-curves of bidegree $(5,5)$ on $X$ realizing all real schemes listed above, with $3 \leq i \leq 17$.
\end{proof}
\subsection{Structure of the paper}
\label{subsec: structure}
In Section \ref{sec: restrictions}, we present previously known results that give topological restrictions to type II, resp. type I, real algebraic $(M-i)$-curves of bidegree $(d,d)$ on $X$. The content of this paper is the proof that such system of restrictions is complete for bidegree $(5,5)$. Indeed, the main part of the paper concerns the construction of type II, resp. type I, real algebraic $(M-i)$-curves of bidegree $(5,5)$ in $X$ realizing all real schemes which are not prohibited by the restrictions in Section \ref{sec: restrictions}. \\
In Section \ref{subsec: the quadratic cone}, we explain how to construct a real algebraic curve of bidegree $(d,d)$ on the quadric ellipsoid in $\mathbb{C}P^3$ with topology prescribed by the topology of a real algebraic curve of bidegree $(d,0)$ constructed on the second Hirzebruch surface. In the rest of Section \ref{sec: methods_construction}, we introduce Hirzebruch surfaces (Section \ref{subsec: Hirzebruch_surf}) and we give some construction tools we have on such surfaces: 
\begin{itemize}
\item particular birational transformations (Section \ref{subsec: Hirzebruch_surf});
\item Orevkov's method via dessins d'enfants (Section \ref{subsec: dess_enf}).
\end{itemize}
In Section \ref{sec: construction} we end the proof of Theorems \ref{thm: maximal_thm}, \ref{thm: m-1_curves_thm}, \ref{thm: type_I,II_m-2_thm} and \ref{thm: type_I_II_thm}:
\begin{itemize}
\item in Section \ref{subsec: trigonal_construction}, we realize some intermediate constructions on Hirzebruch surfaces using the construction tools of Section \ref{sec: methods_construction};
\item in Section \ref{subsec: patch}, using the Viro's patchworking method (\cite{Viro06}) and a variant of it developed by Shustin (\cite{Shus05}) we finally construct real algebraic curves of bidegree $(5,5)$ on $X$ realizing all real schemes listed in Theorems \ref{thm: maximal_thm}, \ref{thm: m-1_curves_thm}, \ref{thm: type_I,II_m-2_thm} and \ref{thm: type_I_II_thm}.
\end{itemize}
\subsection{Acknowledgments}
I am very grateful to Erwan Brugallé for the support and numerous fruitful discussions throughout my journey to discover many aspects of research regarding the classification of real algebraic curves. I extend my gratitude to Andrés Jaramillo Puentes who was very helpful in my understanding in dessins d'enfants. I would like to thank the referee for the constructive and useful remarks.
\section{Restrictions}
\label{sec: restrictions}
\subsection{Restriction on the depth of nests}
\label{subsec: bezout}
A collection $N_h$ of $h$ disjoint embedded circles in $S^2$ is called \emph{a nest of depth $h$} if any connected component of $S^2\setminus N_h$ is either a disc or an annulus. Two nests are said to be disjoint if each of them lies on one of the discs bounded by the other.
\begin{prop} \cite[Proposition $4.9.2$]{DegKha00}
\label{prop: bezout}
Let $A$ be a non-singular real algebraic curve of bidegree $(d,d)$ on $X$. Then
the total number of ovals in any collection of three pairwise disjoint nests of $\mathbb{R} A$ does not exceed $d$.
\end{prop}
Proposition \ref{prop: bezout} implies in particular that the maximal depth for a nest of such a curve $A$ is $d$. Furthermore, it is well known that if $A$ is of type I and has $d$ ovals, it has a nest of maximal depth $d$ (see Proposition \ref{prop: orientation_formula}).
\\
Combining Proposition \ref{prop: Harnack-Klein} with Proposition \ref{prop: bezout}, one obtains the following immediate result.
\begin{cor}
\label{cor: bezout_(5,5)}
Let $A$ be a non-singular real algebraic curve of bidigree $(5,5)$ on $X$. Then the pair $(\mathbb{R}A,\mathbb{R}X)$ realizes one of the following real schemes:
\begin{itemize}
\item $0$ and $1,$
\item $\alpha \quad \sqcup \quad \langle \beta  \rangle \quad \sqcup \quad \langle \gamma  \rangle, \text{  } \text{for} \text{ } 0 \leq \alpha+\beta+\gamma \leq 15,$
\item $\langle\langle\langle \langle 1 \rangle  \rangle  \rangle \rangle.$
\end{itemize}
\end{cor}
\subsection{Congruences and complex orientation formula on the ellipsoid}
\label{subsec: mikha's_th_cx_orient_formula}
As an application of Guillou-Marin congruences \cite{GuiMar77} (which are a generalization of Rokhlin's congruences), 
Mikhalkin proved the following theorem.
\begin{thm}\cite[Theorem $1$]{Mikh94}
\label{thm: mikha_congruences}
Let $A$ be a non-singular real algebraic $(M-i)$-curve of bidegree $(d,d)$ on $X$, with $d$ odd. Let $B$ be a disjoint union of connected components of $\mathbb{R}X\setminus \mathbb{R}A$ such that $\mathbb{R}A$ bounds $B$.
\begin{enumerate}[label=(\arabic*)]
\item If $A$ is a $M$-curve, then
$$\chi (B)\equiv \frac{d^2 + 1}{2} \pmod{8}.$$
\item If $A$ is a $(M-1)$-curve, then
 $$\chi (B)\equiv \frac{d^2 + 1}{2} \pm 1 \pmod{8}.$$
	\item If $A$ is a $(M-2)$-curve and 
$$\chi (B)\equiv \frac{d^2 -7}{2} \pmod{8},$$
then $A$ is of type I.\\
\item If $A$ is of type I, then 
$$\chi (B)\equiv 1 \pmod{4}.$$
\end{enumerate}
\end{thm}
Let $A$ be a non-singular real algebraic curve of type I on $X$. Fix a complex orientation on $\mathbb{R}A$. Since any pair of ovals of $\mathbb{R}A$ bounds an annulus in $\mathbb{R}X$, we distinguish two types of pairs: denote by $\Pi_-$ (respectively $\Pi_+$) the number of pairs of ovals realizing the same (resp. different) first homology class of the corresponding annulus. As an application of the generalizations of Rokhlin’s formula of complex orientations, Zvonilov in \cite{Zvon83} gave a complex orientation formula for type I non-singular real algebraic curves on $X$. This formula depends on the choice of an auxiliary point in $\mathbb{R}X\setminus \mathbb{R}A$. Afterwards, Orekvov in \cite{Orev207} reformulated it with no dependence on the choice of an auxiliary point. 
\begin{prop} \cite{Zvon83}, \cite[Proposition 1.2]{Orev207}
\label{prop: orientation_formula}
Let $A$ be a non-singular real algebraic type I curve of bidegree $(d,d)$ on $X$. Denoting by $l$ the number of connected components of $\mathbb{R}A$, one has the following complex orientation formula:
\begin{equation}
\label{eqn: rokhlin_ellipsoid}
2(\Pi_+ - \Pi_-)=l-d^2
\end{equation}
\end{prop}
\begin{cor} \cite[Proposition $1.3$]{Orev207}
\label{cor: type_I_tipi}
Let $A$ be a non-singular real algebraic type I curve of bidegree $(d,d)$ on $X$. Then $\mathbb{R}A$ has at least $d$ connected components. Furthermore, in the case where $\mathbb{R}A$ has $d$ connected components, it consists of a nest of maximal depth $d$.
\end{cor}
Corollary \ref{cor: bezout_(5,5)} and Theorem \ref{thm: mikha_congruences} give a complete system of restrictions for real schemes of non-singular real algebraic curves of bidegree $(5,5)$ on $X$. Moreover, Theorem \ref{thm: mikha_congruences} and Proposition \ref{prop: orientation_formula} allow us to give even finer restrictions on which real schemes, listed in Corollary \ref{cor: bezout_(5,5)}, may be realized by type I (resp. type II) non-singular real algebraic curves of bidegree $(5,5)$ on $X$. Therefore, given a non-singular real algebraic curve $A$ of bidegree $(5,5)$ on $X$, the pair $(\mathbb{R}X, \mathbb{R}A)$ realizes one of the real schemes listed in Theorems \ref{thm: maximal_thm}, \ref{thm: m-1_curves_thm}, \ref{thm: type_I,II_m-2_thm}, \ref{thm: type_I_II_thm}. \\
In the next sections we pass to the construction part of the classification.
\section{Construction tools}
\label{sec: methods_construction}
\subsection{Hirzebruch surfaces}
\label{subsec: Hirzebruch_surf}
A Hirzebruch surface is a compact complex surface which admits a holomorphic fibration over $\mathbb{C}P^1$ with fiber $\mathbb{C}P^1$. Every Hirzebruch surface is biholomorphic to exactly one of the surfaces $\Sigma_n=\mathbb{P}(\mathcal{O}_{\mathbb{C}P^1}(n)\oplus \mathbb{C})$ for $n \geq 0$. The surface $\Sigma_n$ admits a natural fibration $$\pi_n: \Sigma_n \rightarrow \mathbb{C}P^1$$ with fiber $\mathbb{C}P^1=:F_n$. Denote by $B_n$, resp. $E_n$, the section $\mathbb{P}(\mathcal{O}_{\mathbb{C}P^1}(n)\oplus \{0\})$, resp. $\mathbb{P}(\{0\}\oplus \mathbb{C})$. The self-intersection of $B_n$ (resp. $E_n$ and $F_n$) is $n$ (resp. $-n$ and $0$). When $n \geq 1$, the exceptional divisor $E_n$ determines uniquely the Hirzebruch surface since it is the only irreducible and reduced algebraic curve in $\Sigma_n$ with negative self-intersection.\\
\par For example $\Sigma_0=\mathbb{C}P^1\times \mathbb{C}P^1$. The Hirzebruch surface $\Sigma_1$ is the complex projective plane blown-up at a point, and $\Sigma_2$ is the quadratic cone with equation $Q_0: X^2+ Y^2-Z^2=0$ blown-up at the node in $\mathbb{C}P^3$. The fibration of $\Sigma_2$ (resp. of $\Sigma_1$) is the extension of the projection from the blown-up point to a hyperplane section (resp. to a line) which does not pass through the blown-up point.\\
\par The group $H_2(\Sigma_n; \mathbb{Z})$ is isomorphic to $\mathbb{Z} \oplus \mathbb{Z}$ and is generated by the classes $[ B_n ]$ and $[ F_n ]$. An algebraic curve C in $\Sigma_n$ is said to be of bidegree $(a,b)$ if it realizes the homology class $a[ B_n ] + b[ F_n ]$ in $H_2(\Sigma_n; \mathbb{Z})$. Note that $[E_n]=[B_n]-n[F_n]$ in $H_2(\Sigma_n; \mathbb{Z})$. An algebraic curve of bidegree $(3,0)$ on $\Sigma_n$ is called a \textit{trigonal curve}.
\par We can obtain $\Sigma_{n+1}$ from $\Sigma_n$ via a birational transformation $\beta^p_n: \Sigma_n - - \rightarrow \Sigma_{n+1}$ which is the composition of a blow-up at a point $p\in E_n \subset \Sigma_n$ and a blow-down of the strict transform of the fiber $\pi^{-1}_n(\pi_n(p))$.\\ 
\par The surface $\Sigma_n$ is also the projective toric surface which corresponds to 
the polygon of vertices $(0,0), (0,1),$ $(1,1), (n+1,0)$, depicted in Fig. \ref{fig: polygon_sigma_n} $a)$ where the number labeling an edge corresponds to its integer length. 
The Newton polygon of an algebraic curve $C$ of bidegree $(a,b)$ on $\Sigma_n$, lies inside the trapeze with vertices $(0,0), (0,a), (b,a), (an+b,0)$ as in Fig. \ref{fig: polygon_sigma_n} $b)$. 
The surface $\Sigma_n$ is canonically endowed by a real structure induced by the standard complex conjugation in $(\mathbb{C^*})^2$. 
For this real structure the real part of $\Sigma_n$, denoted by $\mathbb{R}\Sigma_n$, is a torus if $n$ is even and a Klein bottle if $n$ is odd. We will depict $\mathbb{R}\Sigma_n$ as a quadrangle whose opposite sides are identified in a suitable way. Moreover, the horizontal sides will represent $\mathbb{R}E_n$. \\
\begin{figure} [!h]
\begin{center}
\begin{tikzpicture}[scale=0.8]
\draw (0,0) -- (0,3) -- (1,3) -- (4,0) -- (0,0);
\node at (-0.5,1.5) {$1$};
\node at (0.5,3.5) {$1$};
\node at (2,-0.5) {$n+1$};
\node at (3,2) {$1$};
\node at (2,-1) {$a)$ Polygon defining $\Sigma_n$};
\node at (2,-1.8) { };
\end{tikzpicture}
\begin{tikzpicture}[scale=0.8]
\draw (0,0) -- (0,3) -- (1,3) -- (4,0) -- (0,0);
\node at (-0.5,1.5) {$a$};
\node at (0.5,3.5) {$b$};
\node at (2,-0.5) {$an+b$};
\node at (3,2) {$a$};
\node at (2,-1) {$b)$ The Newton polygon of a curve of };
\node at (2,-1.8) {bidegree $(a,b)$ on $\Sigma_n$};
\end{tikzpicture}
\end{center}
\caption{ }
\label{fig: polygon_sigma_n}
\end{figure}
\par The restriction of $\pi_n$ to $\mathbb{R}\Sigma_n$ defines an $S^1$-bundle over $S^1$ that we denote by $\mathcal{L}$. We are interested in the isotopy types with respect to $\mathcal{L}$ of real algebraic curves in $\mathbb{R}\Sigma_n$. 
\begin{defn}
\begin{itemize}
\item []
\item The topological type $\eta$ realized by a pair $(\mathbb{R} \Sigma_{n}, \mathcal{B} )$, where $\mathcal{B}$ is a union of circles and points immersed in $\mathbb{R}\Sigma_n$, is called a real scheme. We say that a real algebraic curve $C \in \mathbb{R}\Sigma_n$ has real scheme $\eta$ if the pair $(\mathbb{R}\Sigma_n, \mathbb{R}C)$ realizes $\eta$.
\item Two real schemes in $\mathbb{R}\Sigma_n$ are \textit{$\mathcal{L}$-isotopic} if there exists an isotopy of $\mathbb{R}\Sigma_n$ which brings one arrangement to the other, each line of $\mathcal{L}$ to another line of $\mathcal{L}$ and whose restriction to $\mathbb{R}E_n$ is an isotopy of $\mathbb{R}E_n$. 
\item A real scheme in $\mathbb{R}\Sigma_n$ up to $\mathcal{L}$-isotopy of $\mathbb{R}\Sigma_n$ is called an \textit{$\mathcal{L}$-scheme}.
\item An $\mathcal{L}$-scheme is \textit{realizable} by a real algebraic curve of bidegree $(a,b)$ in $\Sigma_n$ if there exists such a curve whose real scheme is $\mathcal{L}$-isotopic to the arrangement of circles and points in $\mathbb{R}\Sigma_n$. 
\item A \textit{trigonal} $\mathcal{L}$-scheme is an $\mathcal{L}$-scheme in $\mathbb{R}\Sigma_n$ which intersects each fiber in 1 or 3 real points counted with multiplicities and which does not intersect $\mathbb{R}E_n$.
\item A trigonal $\mathcal{L}$-scheme $\eta$ in $\mathbb{R}\Sigma_n$ is \textit{hyperbolic} if it intersects each fiber in $3$ real points counted with multiplicities.
\end{itemize}
\end{defn}
\subsection{The quadric ellipsoid and the second Hirzebruch surface}
\label{subsec: the quadratic cone}
We explain in this section how to construct a real algebraic curve of bidegree $(d,d)$ on the quadric ellipsoid in $\mathbb{C}P^3$ with topology prescribed by the topology of a real algebraic curve of bidegree $(d,0)$ in $\Sigma_2$, endowed with the canonical real structure. The advantage of working with the real toric surface $\Sigma_2$ is that we dispose on such surface of different construction tools such as: the birational transformations presented in Section \ref{subsec: Hirzebruch_surf}, Orevkov's construction method via dessins d'enfants (Section \ref{subsec: dess_enf}) and Viro's patchworking method (see for example \cite{Viro06}).\\
Let $C$ be a real algebraic curve of degree $(d,0)$ in $\Sigma_2$. Let $\eta$ be the real scheme realized by $\mathbb{R}C$ in $\mathbb{R}\Sigma_2$ (a torus). Now, cut $\mathbb{R}\Sigma_2$ along $\mathbb{R}E_2$, as depicted in $a)$ of Fig. \ref{fig: cutting_torus_to_sphere}, and glue two discs $D_1$, $D_2$ as depicted in $b)$ of Fig. \ref{fig: cutting_torus_to_sphere}. By this construction we obtain a $2$-sphere $S^2$. Moreover, from the arrangement of the triplet $(\mathbb{R}\Sigma_2, \mathbb{R}E_2, \eta)$ we obtain an arrangement $B$ of embedded circles in $S^2$. As example, look at Fig \ref{fig: example_cutting_torus_to_sphere} where we obtain the arrangement $1 \quad \sqcup \quad \langle 1 \rangle$ in $S^2$.\\
\begin{figure}[h!]
\begin{center}
\begin{picture}(100,90)
	\put(-20,0){$a)$}
\put(130,0){$b)$}
\put(102,55){\textcolor{black}{$D_2$}}
\put(-120,10){\includegraphics[width=1\textwidth]{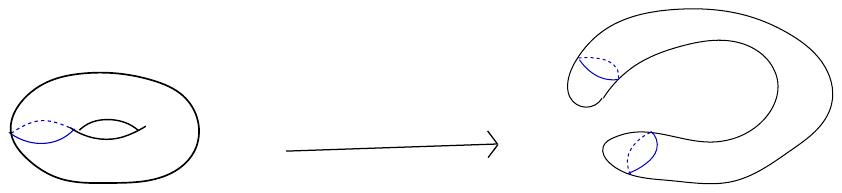}} 
\put(92,75){$D_1$}
\put(-75,55){\textcolor{blue}{$\mathbb{R}E_2$}}
\put(-75,30){$\mathbb{R}\Sigma_2$}
\put(160,30){$S^2$}
\end{picture}
\end{center}
\caption{From a torus to a 2-sphere}
\label{fig: cutting_torus_to_sphere}
\end{figure}
\begin{figure}[!h]
\begin{center}
\begin{picture}(100,90)
	\put(-20,0){$a)$}
\put(130,0){$b)$}
\put(-120,10){\includegraphics[width=1\textwidth]{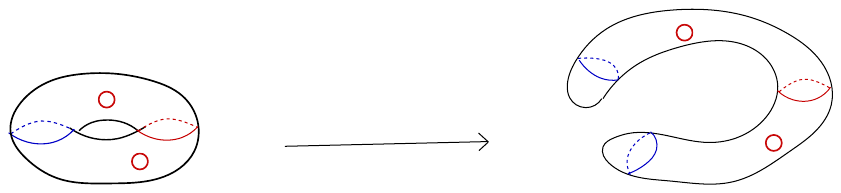}} 
\put(-75,30){$\mathbb{R}\Sigma_2$}
\put(160,20){$S^2$}
\end{picture}
\end{center}
\caption{Example: from an arrangement of embedded circles in $\mathbb{R}\Sigma_2$ to an arrangement in $S^2$}
\label{fig: example_cutting_torus_to_sphere}
\end{figure}
\begin{prop}
\label{prop: construct_on_torus_to_construct_on_ellipsoid}
Let $C$ be a non-singular real algebraic curve of bidegree $(d,0)$ in $\Sigma_2$. Let $B$ be the real scheme on the sphere $S^2$ obtained from the pair $(\mathbb{R}\Sigma_2, \mathbb{R}C)$ by the construction above. Then, $B$ is realizable by a non-singular real curve of bidegree $(d,d)$ on the quadric ellipsoid in $\mathbb{C}P^3$.
\end{prop}
\begin{figure}[h!]
\begin{center}
\begin{picture}(100,100)
	\put(-20,0){$a)$}
\put(60,0){$b)$}
\put(130,0){$c)$}
\put(-70,-10){\includegraphics[width=0.70\textwidth]{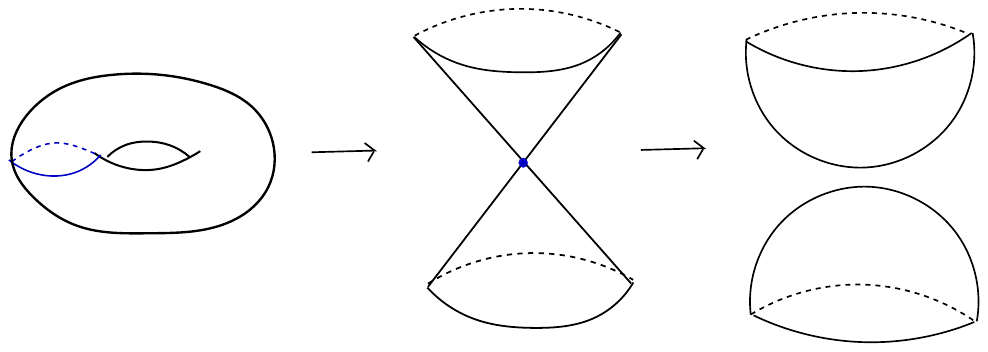}} 
\put(-65,50){\textcolor{blue}{$\mathbb{R}E_2$}}
\put(-55,30){$\mathbb{R}\Sigma_2$}
\put(160,30){$\mathbb{R}Q_{\varepsilon}$}
\put(60,50){\textcolor{blue}{$q$}}
\put(85,30){\textcolor{black}{$\mathbb{R}Q_0$}}
\end{picture}
\end{center}
\caption{ }
\label{fig: torus_cone_sphere}
\end{figure}
\begin{proof}
Let $[X:Y:Z:W]$ be the homogeneous coordinates in the $3$-dimensional complex projective space. Let $Q_0$ be the quadratic cone with equation $X^2+ Y^2-Z^2=0$ in $\mathbb{C}P^3$. 
Recall that we obtain $\Sigma_2$ blowing-up $Q_0$ at the point $q=[0:0:0:1]$.
The image of $C$ via the blow-down is a real algebraic curve $\tilde{C}$ of degree $2d$ in $\mathbb{C}P^3$. Since the dimension of the space of curves of bidegree $(d,0)$ in $\Sigma_2$ is equal to the dimension of the space of complete intersections of the surfaces of degree $d$ in $\mathbb{C}P^3$ with $Q_0$, the curve $\tilde{C}$ is the intersection of a non-singular real algebraic surface $S_d$ of degree $d$, not passing through the node of $Q_0$, and $Q_0$.  
Observe that we can perturb $Q_0$ to the quadric ellipsoid $Q_{\varepsilon}$ of equation $X^2+Y^2-Z^2=-\varepsilon W^2$, where $\varepsilon > 0$ (see Fig. \ref{fig: torus_cone_sphere}). Since a real algebraic curve of bidegree $(d,d)$ on $Q_{\varepsilon}$ is the intersection of the quadric ellipsoid and a surface of degree $d$, the intersection of $S_d$ and $Q_{\varepsilon}$ is a real algebraic curve $A$ of bidegree $(d,d)$. Moreover, the pair $(\mathbb{R}Q_{\varepsilon}, \mathbb{R}A)$ realizes $B$.
\end{proof}

\subsection{Dessins d'enfants}
\label{subsec: dess_enf}
Orevkov in \cite{Orev03} has formulated the existence of real algebraic trigonal curves realizing a given trigonal $\mathcal{L}$-scheme in $\mathbb{R}\Sigma_n$ in terms of the existence of a real rational graph on $\mathbb{C}P^1$. Later on, Degtyarev, Itenberg and Zvonilov in \cite{DegIteZvo14} have given a general way to determine if such real algebraic curves are of type I or II. \\
In Section \ref{subsec: trigonal_construction}, we exploit this construction technique. 
Therefore, we present here some results of \cite{Orev03} and \cite{DegIteZvo14}.
\begin{defn}
\label{defn: complete_graph}
Let $n$ be a fixed positive integer. We say that a graph $\Gamma$ is a \textit{real trigonal graph of degree $n$} if 
\begin{itemize}
\item it is a finite oriented connected graph embedded in $\mathbb{C}P^1$, invariant under the standard complex conjugation of $\mathbb{C}P^1$;
\item it is decorated with the following additional structure:
\begin{itemize}
\item every edge of $\Gamma$ is colored solid, bold or dotted;
\item every vertex of $\Gamma$ is $\bullet$, $\circ$, $\times$ (said \textit{essential vertices}) or monochrome
\end{itemize}
and satisfying the following conditions:
\begin{enumerate}
\item any vertex is incident to an even number of edges; moreover, any $\circ$-vertex (resp. $\bullet$-vertex) to a multiple of $4$ (resp. $6$) number of edges;
\item  for each type of essential vertices, the total sum of edges incident to the vertices of a same type is $12n$;
\item the orientations of the edges of $\Gamma$ form an orientation of $\partial (\mathbb{C}P^1\setminus \Gamma)$ which is compatible with an orientation of $\mathbb{C}P^1\setminus \Gamma$ (see Fig. \ref{fig: vertices_gamma});
\item all edges incidents to a monochrome vertex have the same color;
\item $\times$-vertices are incident to incoming solid edges and outgoing dotted edges;
\item $\circ$-vertices are incident to incoming dotted edges and outgoing bold edges;
\item $\bullet$-vertices are incident to incoming bold edges and outgoing solid  edges.
\end{enumerate}
\end{itemize}
\end{defn}
Let $n$ be a positive integer and let $c(x,y)=y^3+b_2(x)y+b_3(x)$ be a real polynomial, where $b_i(x)$ has degree $in$ in $x$. By a suitable change of coordinates in $\Sigma_n$, any trigonal curve $C$ in $\Sigma_n$ can be put into this form. Denote by $\Delta=-4b_2^3+27b_3^2$ the discriminant of $c(x,y)$ with respect to the variable $y$. The knowledge of the arrangement of the real roots of the real polynomials $\Delta=-4b_2^3+27b_3^2$, $27b_3^2$ and $-4b_2^3$ in $\mathbb{R}\Sigma_n$ allows to recover the trigonal $\mathcal{L}$-scheme realized by $C$ in $\mathbb{R}\Sigma_n$. Let $f: \mathbb{C}P^1 \rightarrow \mathbb{C}P^1$ be the homogenized discriminant, i.e. the rational function defined by $f:=\frac{\Delta}{27b_3^2}$. Orevkov's method allows to construct real polynomials $c(x,y)$ which have prescribed arrangements of the real roots and the construction is based on a consideration of the arrangement of the graph given by $f^{-1}(\mathbb{R}P^1)$ with the coloring and orientation induced by those of $\mathbb{R}P^1$ as depicted in $a)$ of Fig.\ref{fig: coloring_proj_line_locally}. In this section, we only give an algorithmic way to encode any trigonal $\mathcal{L}$-scheme in $\mathbb{R}\Sigma_n$ into a colored oriented graph on $\mathbb{R}P^1 \subset \mathbb{C}P^1$ just looking at the intersections of the fibers of $\mathcal{L}$ with $\eta$; for details see \cite{Bruga07}, \cite{Orev03}.
\begin{figure}[h!]
\begin{center}
\begin{picture}(100,40)
\put(-20,-13){$a)$}
\put(-21,3){\small{$0$}}
\put(-57,3){\small{$\infty$}}
\put(11,3){\small{$1$}}
\put(111,-13){$b)$}
\put(105,32){\rotatebox{90}{$\textcolor{black}{\Big \}}$}}
\put(110,42){$\textcolor{black}{\mathcal{L}_{|_{I}}}$}
\put(-140,-10){\includegraphics[width=1\textwidth]{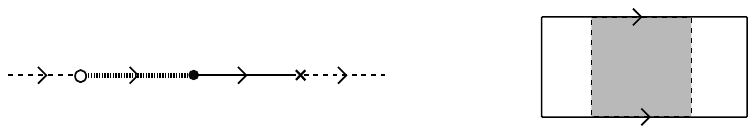}}
\end{picture}
\end{center}
\caption{$a)$ Colored oriented $\mathbb{R}P^1$. $b)$ $\mathcal{L}_{|_{I}}$ in $\mathbb{R}\Sigma_n$.}
\label{fig: coloring_proj_line_locally}
\end{figure} 
\begin{defn}
\label{defn: completable_graph}
\begin{itemize}
Let $\eta$ be a trigonal $\mathcal{L}$-scheme in $\mathbb{R}\Sigma_n$.
For any fixed interval of points $I:=\{(x,\overline{y}): \overline{y} \in \mathbb{R}, \quad x \in [a, a+b] \subset \mathbb{R}, \quad a \not = b\} \subset \mathbb{R}\Sigma_n$, we denote with $\mathcal{L}_{|_{I}}$ the fibers of $\mathcal{L}$ containing the points of $I$ (see $b)$ of Fig. \ref{fig: coloring_proj_line_locally}). Thanks to $\pi_{n|_{\mathbb{R}\Sigma_n}}$ we can encode $\eta$ into a colored oriented graph $\overline{\Gamma}$ on $\mathbb{R}P^1 \subset \mathbb{C}P^1$ as follows (in Fig. \ref{fig: real graph_translate} the dashed lines denote fibers of $\mathcal{L}$):
\begin{enumerate}
\item To each fiber of $\mathcal{L}$ intersecting $\eta$ in two points we associate a $\times$-vertex on $\mathbb{R}P^1$.
\item Fixed an interval $I$, let $F_1,F_2$ be two fibers of $\mathcal{L}_{|_{I}}$ intersecting $\eta$ in two points such that $\eta$, up to $\mathcal{L}$-isotopy, is locally as depicted in $b)$ or $c)$ of Fig. \ref{fig: real graph_translate}. Let $F_3$ be another fiber between $F_1, F_2$. Then, we associate to $F_3$ a $\circ$-vertex on $\mathbb{R}P^1$. Moreover, if between $F_1$ and $F_2$ each other fiber intersects $\eta$ in only one real point (as in $b)$ of Fig. \ref{fig: real graph_translate}), then we associate to a fiber between $F_1$ and $F_3$ (resp. $F_3$ and $F_2$) a $\bullet$-vertex on $\mathbb{R}P^1$. Between $\bullet$ and $\circ$-vertices we put bold edges.
\item For all intervals $I$, except for the fibers of $\mathcal{L}_{|_{I}}$ to which we associate essential vertices and bold edges, we associate dotted (resp. solid) edges on $\mathbb{R}P^1$ to the fibers of $\mathcal{L}_{|_{I}}$ which intersect $\eta$ in three distinct real points (resp. only one real point).
\item The orientations of the edges incident to a vertex are in an alternating order. In particular, the orientations of the edges incident to an essential vertex are respectively as described in $5. ,6. ,7.$ of Definition \ref{defn: complete_graph}.
\end{enumerate}
The graph $\overline{\Gamma}$, called real graph, is considered up to isotopy of $\mathbb{R}P^1$, namely it is determined by the order of its colored vertices since the edges are determined by the color of their adjacent vertices.\\
We say that $\overline{\Gamma}$ is completable in degree $n$ if there exists a complete real trigonal graph $\Gamma$ of degree $n$ such that $\Gamma \cap \mathbb{R}P^1=\overline{\Gamma}$. 
\end{itemize}
\end{defn}
\begin{figure}[h!]
\begin{center}
\begin{picture}(100,90)
\put(-77,-13){$a)$}
\put(-20,-13){$b)$}
\put(54,-13){$c)$}
\put(-43,90){\small{$F_1$}}
\put(-20,90){\small{$F_3$}}
\put(0,90){\small{$F_2$}}
\put(111,-13){$d)$}
\put(40,90){\small{$F_1$}}
\put(50,90){\small{$F_3$}}
\put(60,90){\small{$F_2$}}
\put(184,-13){$e)$}
\put(-120,-50){\includegraphics[width=1\textwidth]{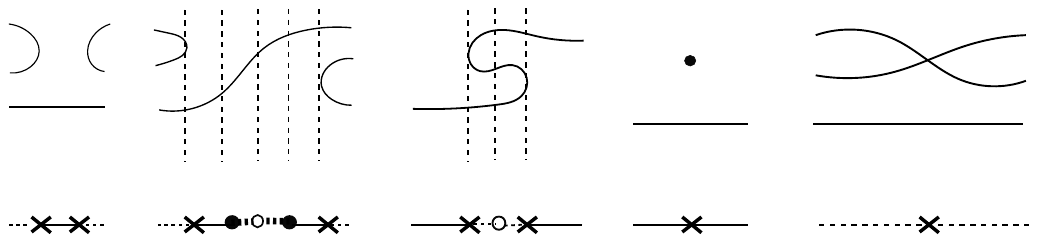}}
\end{picture}
\end{center}
\caption{Local topology of trigonal $\mathcal{L}$-schemes and their corresponding real graphs.}
\label{fig: real graph_translate}
\end{figure} 
\begin{thm}[\cite{Orev03}]
\label{thm: existence_trigonal}
A trigonal $\mathcal{L}$-scheme on $\mathbb{R}\Sigma_n$ is realizable by a real algebraic trigonal curve if and only if its real graph is completable in degree $n$. 
\end{thm}
Given a real graph $\overline{\Gamma}$, we  depict only the completion to a real trigonal graph $\Gamma$ on a hemisphere of $\mathbb{C}P^1$ since $\Gamma$ is symmetric with respect to the standard complex conjugation. Moreover, we can omit orientations in figures representing real trigonal graphs because each vertex is adjacent to an even number of edges oriented in an alternating order as, for example, depicted in Fig. \ref{fig: vertices_gamma} and such orientations are compatible with each others.\\
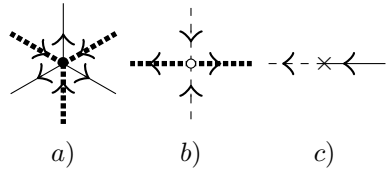
\begin{figure}[h!]
\begin{center}
\begin{tikzpicture} [scale=0.4]
\node at (0,-3) {$a)$};
\node at (0:0) {$\bullet$};
\draw  [densely dotted, line width=0.07cm] (0:0) -- (30:2);
\draw [decoration={markings,mark=at position 1 with {\arrow[scale=2.5]{<}}},
    postaction={decorate},
    shorten >=0.1pt] (30:0.8) -- (30:1);
\draw (0:0) -- (90:2);
\draw [decoration={markings,mark=at position 1 with {\arrow[scale=2.5]{>}}},
    postaction={decorate},
    shorten >=0.1pt] (90:0.8) -- (90:1);
\draw  [densely dotted, line width=0.07cm] (0:0) -- (150:2);
\draw [decoration={markings,mark=at position 1 with {\arrow[scale=2.5]{<}}},
    postaction={decorate},
    shorten >=0.1pt] (150:0.8) -- (150:1);
\draw (0:0) -- (210:2);
\draw [decoration={markings,mark=at position 1 with {\arrow[scale=2.5]{>}}},
    postaction={decorate},
    shorten >=0.1pt] (210:0.8) -- (210:1);
\draw  [densely dotted, line width=0.07cm] (0:0) -- (270:2);
\draw [decoration={markings,mark=at position 1 with {\arrow[scale=2.5]{<}}},
    postaction={decorate},
    shorten >=0.1pt] (270:0.8) -- (270:1);
\draw (0:0) -- (-30:2);
\draw [decoration={markings,mark=at position 1 with {\arrow[scale=2.5]{>}}},
    postaction={decorate},
    shorten >=0.1pt] (-30:0.8) -- (-30:1);
\node at (0:0) {$\bullet$};
\end{tikzpicture}
\begin{tikzpicture}[scale=0.4]
\node at (0,-3) {$b)$};
\draw  [densely dotted, line width=0.07cm] (0:0) -- (0:2);
\draw [decoration={markings,mark=at position 1 with {\arrow[scale=2.5]{>}}},
    postaction={decorate},
    shorten >=0.1pt] (0:1);
\draw [dashed] (0:0) -- (90:2);
\draw [decoration={markings,mark=at position 1 with {\arrow[scale=2.5]{<}}},
    postaction={decorate},
    shorten >=0.1pt] (90:1) -- (90:1.1);
\draw [densely dotted, line width=0.07cm] (0:0) -- (180:2);
\draw [decoration={markings,mark=at position 1 with {\arrow[scale=2.5]{<}}},
    postaction={decorate},
    shorten >=0.1pt] (180:1);
\draw [dashed] (0:0) -- (-90:2);
\draw [decoration={markings,mark=at position 1 with {\arrow[scale=2.5]{<}}},
    postaction={decorate},
    shorten >=0.1pt] (-90:1) -- (-90:1.1);
\draw [fill=white] (0,0) circle (1.5mm);
\end{tikzpicture}
\begin{tikzpicture}[scale=0.4]
\node at (0,-3) {$c)$};
\node at (0:0) {$\times$};
\draw (0:0) -- (0:2);
\draw [decoration={markings,mark=at position 1 with {\arrow[scale=2.5]{<}}},
    postaction={decorate},
    shorten >=0.1pt] (0:1);
\draw [dashed] (0,0) -- (0:-2);
\draw [decoration={markings,mark=at position 1 with {\arrow[scale=2.5]{<}}},
    postaction={decorate},
    shorten >=0.1pt] (0:-1);
\end{tikzpicture}
\end{center}
\caption{Colored vertices of a real trigonal graph.}
\label{fig: vertices_gamma}
\end{figure}
Theorem \ref{thm: existence_trigonal} is improved in \cite{DegIteZvo14} in order to check if a given trigonal $\mathcal{L}$-scheme is realizable by a real trigonal curve of type I. We say that a real algebraic singular curve is of type I (resp. of type II) if its normalization is of type I (resp. of type II).
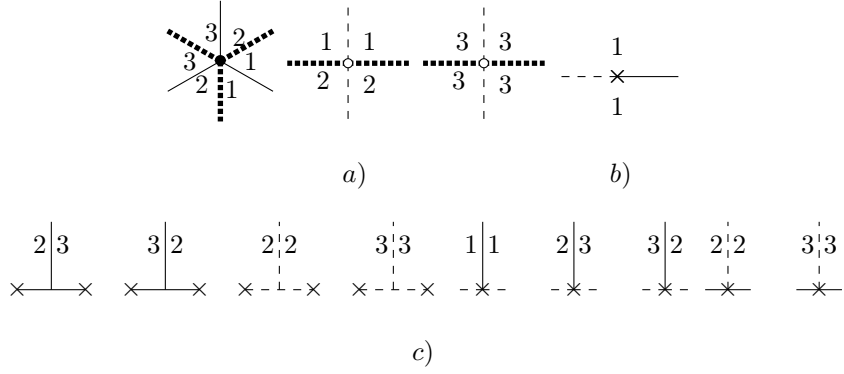
\begin{figure}[h!]
\begin{center}
\begin{tikzpicture} [scale=0.4]
\node at (0:0) {$\bullet$};
\node at (0:1) {\textcolor{black}{1}};
\draw  [densely dotted, line width=0.07cm] (0:0) -- (30:2);
\node at (52:1) {\textcolor{black}{2}};
\draw (0:0) -- (90:2);
\node at (107:1) {\textcolor{black}{3}};
\draw  [densely dotted, line width=0.07cm] (0:0) -- (150:2);
\node at (180:1) {\textcolor{black}{3}};
\draw (0:0) -- (210:2);
\node at (232:1) {\textcolor{black}{2}};
\draw  [densely dotted, line width=0.07cm] (0:0) -- (270:2);
\node at (292:1) {\textcolor{black}{1}};
\draw (0:0) -- (-30:2);
\node at (0:0) {$\bullet$};
\end{tikzpicture}
\begin{tikzpicture}[scale=0.4]
\draw  [densely dotted, line width=0.07cm] (0:0) -- (0:2);
\node at (45:1) {\textcolor{black}{1}};
\draw [dashed] (0:0) -- (90:2);
\node at (135:1) {\textcolor{black}{1}};
\draw  [densely dotted, line width=0.07cm] (0:0) -- (180:2);
\node at (215:1) {\textcolor{black}{2}};
\draw [dashed] (0:0) -- (-90:2);
\node at (-45:1) {\textcolor{black}{2}};
\draw [fill=white] (0,0) circle (1.5mm);
\end{tikzpicture}
\begin{tikzpicture}[scale=0.4]
\draw  [densely dotted, line width=0.07cm] (0:0) -- (0:2);
\node at (45:1) {\textcolor{black}{3}};
\draw [dashed] (0:0) -- (90:2);
\node at (135:1) {\textcolor{black}{3}};
\draw  [densely dotted, line width=0.07cm] (0:0) -- (180:2);
\node at (215:1) {\textcolor{black}{3}};
\draw [dashed] (0:0) -- (-90:2);
\node at (-45:1) {\textcolor{black}{3}};
\draw [fill=white] (0,0) circle (1.5mm);
\end{tikzpicture}
\begin{tikzpicture}[scale=0.4]
\node at (0:0) {$\times$};
\draw (0:0) -- (0:2);
\node at (90:1) {\textcolor{black}{1}};
\draw [dashed] (0:0) -- (180:2);
\node at (-90:1) {\textcolor{black}{1}};
\node at (0:0) {$\times$};
\end{tikzpicture}
\end{center}
\begin{tikzpicture}[yscale=2]
\node at (13,0) {  };
\node at (18.5,0) {$a)$};
\node at (22,0) {$b)$};
\end{tikzpicture}
\begin{center}
\begin{tikzpicture}[scale=0.3]
\draw (1.5,0) -- (1.5,3);
\draw  (0,0) -- (3,0);
\node at (0,0) {$\times$};
\node at (3,0) {$\times$};
\node at (1,2) {$2$};
\node at (2,2) {$3$};

\begin{scope}[xshift=5cm]
\draw (1.5,0) -- (1.5,3);
\draw  (0,0) -- (3,0);
\node at (0,0) {$\times$};
\node at (3,0) {$\times$};
\node  at (1,2) {$3$};
\node  at (2,2) {$2$};
\end{scope}

\begin{scope}[xshift=5cm]
\node at (5,0) {$\times$};
\node at (8,0) {$\times$};
\draw [dashed] (5,0) -- (8,0);
\draw [dashed] (6.5,0) -- (6.5,3);
\node  at (6,2) {$2$};
\node at (7,2) {$2$};
\end{scope}

\begin{scope}[xshift=10cm]
\node at (5,0) {$\times$};
\node at (8,0) {$\times$};
\draw [dashed] (5,0) -- (8,0);
\draw [dashed] (6.5,0) -- (6.5,3);
\node at (6,2) {$3$};
\node  at (7,2) {$3$};
\end{scope}
\end{tikzpicture}
\begin{tikzpicture} [scale=0.3]
\draw (10,0) -- (10,3);
\draw [dashed]  (11,0) -- (10,0);
\node at (10,0) {$\times$};
\draw [dashed]  (9,0) -- (10,0);
\node  at (9.5,2) {$1$};
\node  at (10.5,2) {$1$};
\begin{scope}[xshift=4cm] 
\draw (10,0) -- (10,3);
\draw [dashed]  (11,0) -- (10,0);
\node at (10,0) {$\times$};
\draw [dashed]  (9,0) -- (10,0);
\node  at (9.5,2) {$2$};
\node  at (10.5,2) {$3$};
\end{scope}
\begin{scope}[xshift=8cm]
\draw (10,0) -- (10,3);
\draw [dashed]  (11,0) -- (10,0);
\node at (10,0) {$\times$};
\draw [dashed]  (9,0) -- (10,0);
\node  at (9.5,2) {$3$};
\node  at (10.5,2) {$2$};
\end{scope}
\end{tikzpicture}
\begin{tikzpicture} [scale=0.3]
\draw [dashed] (10,0) -- (10,3);
\draw  (11,0) -- (10,0);
\node at (10,0) {$\times$};
\draw   (9,0) -- (10,0);
\node  at (9.5,2) {$2$};
\node  at (10.5,2) {$2$};
\begin{scope}[xshift=4cm] 
\draw [dashed] (10,0) -- (10,3);
\draw  (11,0) -- (10,0);
\node at (10,0) {$\times$};
\draw  (9,0) -- (10,0);
\node  at (9.5,2) {$3$};
\node  at (10.5,2) {$3$};
\end{scope}
\end{tikzpicture}
\end{center}
\begin{tikzpicture}[yscale=2]
\node at (13,0) {  };
\node at (19.4,0) {$c)$};
\end{tikzpicture}
\caption{Type I labeling.}
\label{fig: numerazione_typeI}
\end{figure}
\begin{defn}
\label{defn: numerazione_typeI}
Let $\Gamma$ be a real trigonal graph of degree $n$. We say that $\Gamma$ is of \textit{type I} if we can label each connected component of $\mathbb{C}P^1 \setminus \Gamma$, with the numbers $1,2$ or $3$, such that: 
\begin{itemize}
\item neighboring connected components of a $\bullet$-vertex, or a $\circ$-vertex of $\Gamma$, are labeled as depicted in one of the pictures in $a)$ of Fig. \ref{fig: numerazione_typeI};
\item neighboring connected components of a $\times$-vertex, which does not belong to $\Gamma\cap \mathbb{R}P^1$, are labeled as depicted in $b)$ of Fig. \ref{fig: numerazione_typeI};
\item neighboring connected components of $\times$-vertices, belonging to $\Gamma\cap \mathbb{R}P^1$, are labeled as depicted in in $c)$ of Fig. \ref{fig: numerazione_typeI}.
\end{itemize}
Otherwise, we say that $\Gamma$ is of \textit{type II}.
\end{defn}
The original statement in \cite{DegIteZvo14} of the following theorem treats only the case of non-singular real trigonal curve, but it is possible to extend it to real nodal trigonal curves (\cite{Jara18}).
\begin{thm}[\cite{DegIteZvo14}]
\label{thm: existence_trigonal_type_I}
A non-hyperbolic trigonal $\mathcal{L}$-scheme on $\mathbb{R}\Sigma_n$ is realizable by a real trigonal curve of type I (resp. of type II) if and only if its real graph has a completion in degree $n$ which is of type I (resp. type II).
\end{thm}
\begin{rem}
\label{rem: labeling_just_one_number}
For each non-hyperbolic trigonal $\mathcal{L}$-scheme on $\mathbb{R}\Sigma_n$ realizable by an irreducible real trigonal curve of type I, each completion in degree $n$ of its real graph has a unique type I labeling (see \cite{DegIteZvo14}). So, later on, each time we have to assign a labeling to a real trigonal graph of type I, we could label only one component. 
\end{rem}
\begin{rem}
\label{rem: hyperbolic_type I}
If a hyperbolic trigonal $\mathcal{L}$-scheme in $\mathbb{R}\Sigma_n$ is realizable by a real trigonal curve $C$ in $\Sigma_n$, then the curve $C$ is of type I because the projection $\pi_n: \Sigma_n \rightarrow \mathbb{C}P^1$ (see Section \ref{subsec: Hirzebruch_surf}) gives a totally real morphism on $\mathbb{C}P^1$.
\end{rem}
\subsubsection{Gluing real trigonal graphs}
\label{subsubsec: gluing cubics}
\begin{figure}[h!]
\begin{center}
\begin{picture}(100,145)
\put(-50,130){$2$}
\put(-30,130){$2$} 
\put(-60,120){$3$} 
\put(-10,125){$3$}
\put(-65,95){$3$} 
\put(-13,100){$3$} 
\put(-80,-15){\includegraphics[width=0.7\textwidth]{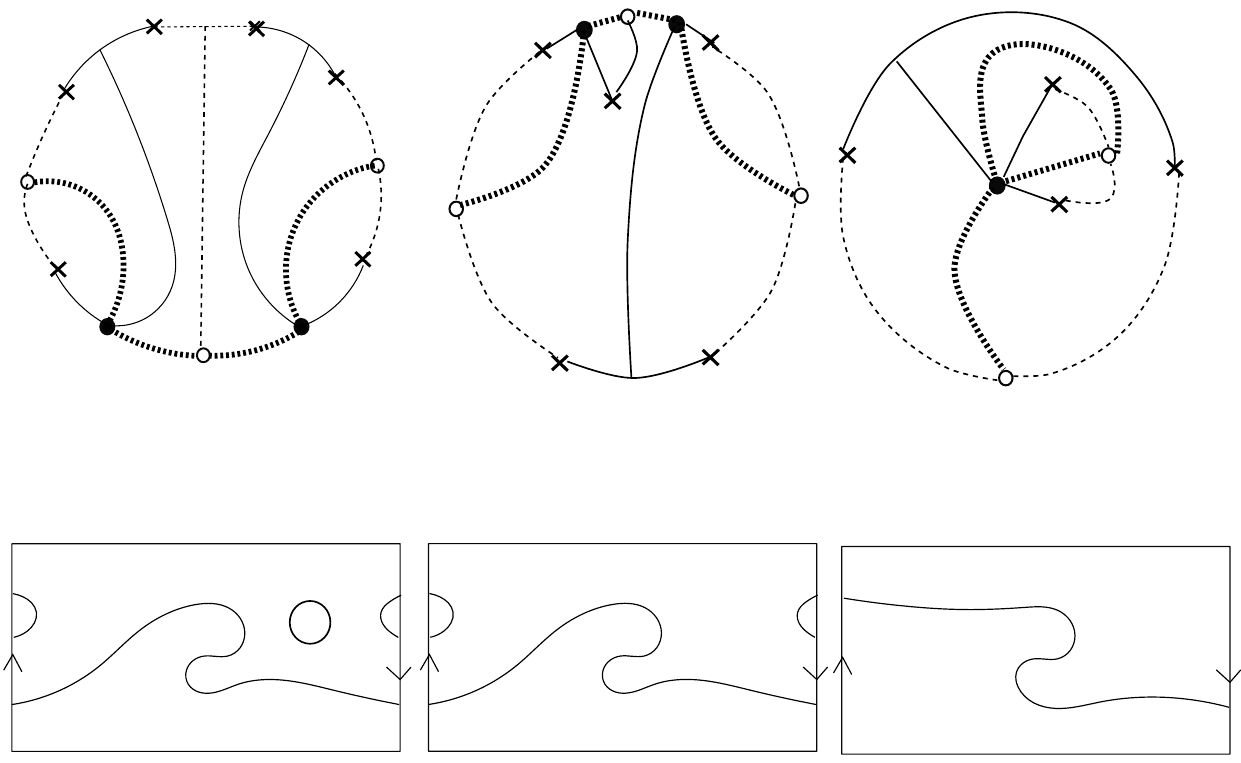}}
\put(-42,-12){$d)$}
\put(-17,-12){$\mathbb{R}\Sigma_1$}
\put(45,-12){$e)$}
\put(70,-12){$\mathbb{R}\Sigma_1$}
\put(130,-12){$f)$}
\put(155,-12){$\mathbb{R}\Sigma_1$}
\put(-42,65){$a)$}
\put(45,65){$b)$}
\put(130,65){$c)$}
\end{picture}
\end{center}
\caption{Cubic trigonal graphs.}
\label{fig: cubic}
\end{figure} 
We call \textit{cubic trigonal graph of type I} (resp. \textit{type II}) a real trigonal graph of degree $1$ and type I (resp. type II). The graph in Fig. \ref{fig: cubic} $a)$ is a cubic trigonal graph of type I, it has a unique type I labeling and associated trigonal $\mathcal{L}$-scheme on $\mathbb{R}\Sigma_1$ as depicted in Fig. \ref{fig: cubic} $d)$. While the graphs depicted in  Fig. \ref{fig: cubic} $b)$ and $c)$ are of type II and have associated trigonal $\mathcal{L}$-schemes on $\mathbb{R}\Sigma_1$ as depicted respectively in Fig. \ref{fig: cubic} $e)$ and $f)$.\\
Let $\Gamma_1$ (resp. $\Gamma_2$) be a real trigonal graph. Denote by $D_1$ (resp. $ D_2$) the disk on which one of the two symmetric halves of $\Gamma_1$ (resp. $ \Gamma_2$) lies. Consider the disjoint union $\Gamma_1 \sqcup \Gamma_2 \subset D_1 \sqcup D_2$. Let $I_i \subset D_i$, $i=1,2$, be a segment in $\mathbb{R}P^1$ whose endpoints are not vertices of $\Gamma_i$ and such that $I_i$ contains a single $\circ$-vertex or a monochrome dashed vertex $\dashedtop$. Let $\phi: I_1 \rightarrow I_2$ be an isomorphism preserving orientation, i.e. $\Gamma_1 \cap I_1\rightarrow \Gamma_2 \cap I_2$ is an isomorphism preserving the types of vertices and edges, and preserving orientation. Consider the quotient $D_1 \sqcup_{\phi} D_2=D_1 \sqcup D_2/(x \sim \phi(x))$ and $\Gamma_{\phi}\subset D_1 \sqcup_{\phi} D_2$ the quotient of the image of $\Gamma_1 \sqcup \Gamma_2$. We call such operation \textit{gluing}. The gluing of real trigonal graphs is still a real trigonal graph (see \cite[Section 5.6]{DegIteKha08} for details).\\
\begin{figure}[t!]
\begin{center}
\begin{picture}(100,70)
\put(-85,59){$3$} 
\put(-85,48){$3$} 
\put(-100,40){$2$} 
\put(-20,40){$2$}
 \put(-65,20){$3$}
 \put(-85,12){$3$}
\put(-10,31){$2$}
\put(-50,30){$2$}
\put(-1,22){$3$}
\put(9,17){$3$} 
\put(12,48){$3$} 
\put(11,60){$3$}     
       \put(-40,10){$a)$}
 
    \put(140,10){$b)$}
         
\put(-110,0){\includegraphics[width=0.9\textwidth]{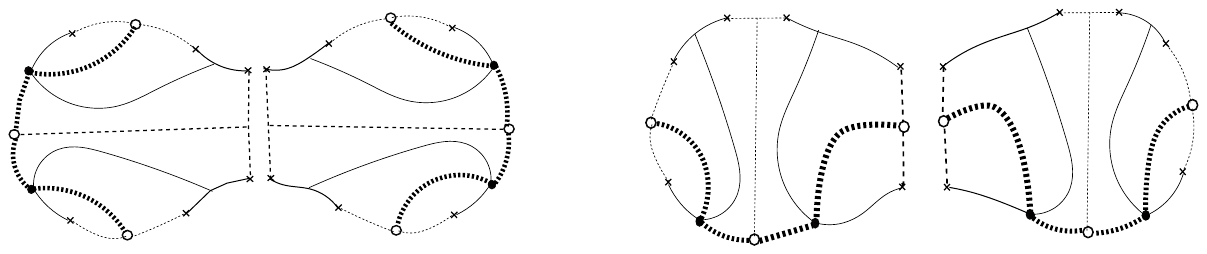}}
\end{picture}
\end{center}
\caption{How gluing two cubic trigonal graphs.}
\label{fig: standard gluing example}
\end{figure} 
\begin{figure}[t!]
\begin{center}
\begin{picture}(100,210)
      \put(-20,110){$a)$}
\put(-85,190){$3$} 
\put(-85,170){$3$} 
\put(-90,160){$2$} 
\put(0,175){$2$}
\put(-60,130){$3$}
\put(-90,128){$3$}
\put(0,140){$2$}
\put(-40,140){$2$}
\put(20,130){$3$}
\put(40,130){$3$} 
\put(40,178){$3$} 
\put(40,192){$3$}     
     \put(130,110){$b)$}
      
         \put(180,120){$\mathbb{R}\Sigma_2$}
  \put(-20,10){$c)$}

    \put(130,10){$d)$}
          \put(180,20){$\mathbb{R}\Sigma_2$}
\put(-110,0){\includegraphics[width=0.9\textwidth]{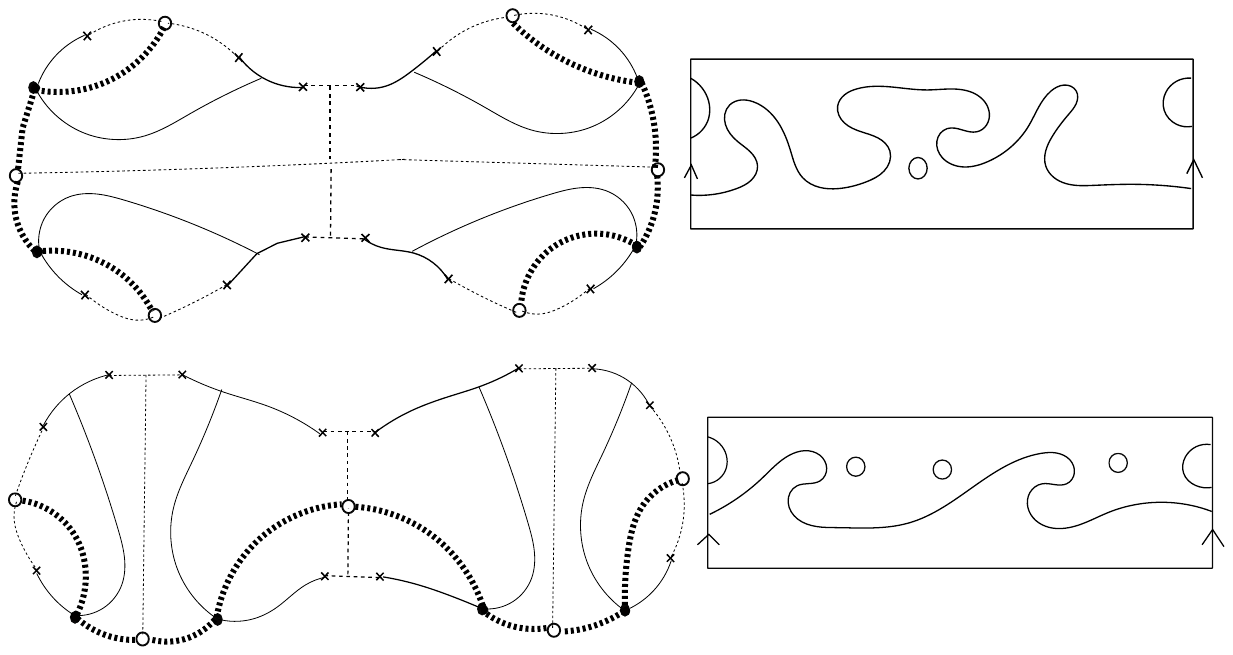}}
\end{picture}
\end{center}
\caption{Gluing of two cubic trigonal graphs and associated trigonal $\mathcal{L}$-scheme.}
\label{fig: 2standard trigonal example}
\end{figure} 
Up to $\mathcal{L}$-isotopy, there is a finite number of possibilities to glue $n$ cubic trigonal graphs into a real trigonal graph of degree $n$. We do not know whether any given trigonal graph of degree $n$ is equivalent to the result of gluing of a union of $n$ cubic trigonal graphs (\cite{Orev03} \cite{DegIteKha08}).\\
Gluing type I real trigonal graphs, which are glued to each other along vertices whose neighboring connected components have the same labels, we get a type I real trigonal graph. As example, look at the gluing of two cubic trigonal graphs of type I in $a)$, resp. $b)$, of Fig. \ref{fig: standard gluing example} and $a)$, resp. $c)$ of Fig. \ref{fig: 2standard trigonal example}. The obtained graphs are real trigonal graphs of degree $2$ and type I. The respective associated trigonal $\mathcal{L}$-schemes are depicted in $b)$ and $d)$ of Fig. \ref{fig: 2standard trigonal example}.  
\section{Construction}
\label{sec: construction}
\subsection{Trigonal construction}
\label{subsec: trigonal_construction}
\begin{figure}[h!]
\begin{center}
\begin{picture}(100,40)
\put(100,-20){\textcolor{black}{$\mathbb{R}\Sigma_5$}}
\put(-50,-30){\includegraphics[width=0.50\textwidth]{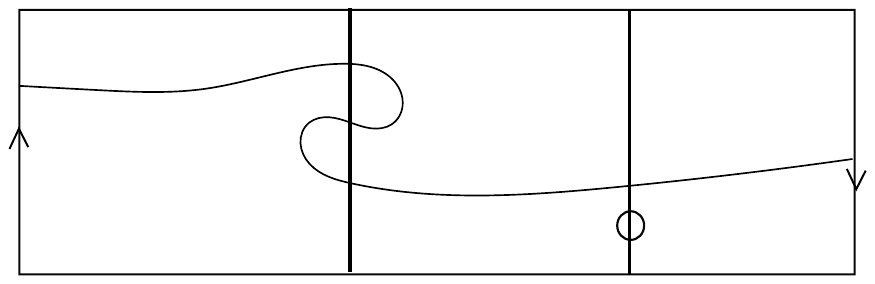}} 
\put(43,15){$a$}
\put(79,15){$b$}
\put(56,-6){$c$}
\put(96,-6){$c'$}
\end{picture}
\end{center}
\caption{The union of a trigonal $\mathcal{L}$-scheme and two fibers of $\mathcal{L}$ on $\mathbb{R}\Sigma_5$: $\eta_{a,b,c,c'}$.}
\label{fig: after_snodato_186_sigma_5}
\end{figure}
In this section we give some intermediate constructions of real algebraic curves that we will need later on.
\begin{prop}
\label{prop: esistenza_in_sigma_5}
Let $\eta_{a,b,c,c'}$ be, up to isotopy of $\mathbb{R}\Sigma_5$, the union of a trigonal $\mathcal{L}$-scheme with two fibers of $\mathcal{L}$ on $\mathbb{R}\Sigma_5$ as depicted in Fig. \ref{fig: after_snodato_186_sigma_5}, where letters $a,b,c,c'$ denote numbers of ovals. Let $h,j$ and $t$ be non-negative integer numbers. Then, there exist real algebraic trigonal curves in $\Sigma_5$ realizing the real schemes $\eta_{a,b,c,c'}$ for all $a,b,c,c'$ such that $0 \leq c+c' \leq h, 0 \leq a \leq j$ and $ 0 \leq b \leq t$, where $h,j$ and $t$ are as follows:
\begin{enumerate}
\item $j+h+t=12$ with 
\begin{itemize}
 \item $h=1$ and $j \in \{0,1,4,7,10,11\}$,
\item $h=5$ and $j\in \{0,1,2,3,4,5,6,7\}$,
\item $h=9$ and $j\in \{0,1,2,3\}$.
\end{itemize}
\item $j+h+t=10$ with
\begin{itemize}
\item $h=0$ and $j\in \{4,6,8\}$,
\item $h=2$ and $j\in \{0,1,2,4,5,6,8\}$,
\item $h=4$ and $j\in \{0,1,2,4,5,6\}$,
\item $h=6$ and $j\in \{0,1,2,4\}$,
\item $h=8$ and $j\in \{0,1,2\}$.
\end{itemize}

\item $j+h+t=8$ with
\begin{itemize}
\item $h=1$ and $j=3$,
\item $h=3$ and $j\in \{1,2,5\}$.
\end{itemize}
\end{enumerate}
Furthermore, the real trigonal curves with $c+c'= h, a = j$ and $ b = t$, are of type I.
Also, there exist real trigonal curves of type I in $\Sigma_5$ realizing $\eta_{a,b,c,c'}$ for $(a,b,c,c')=(1,5,0,0)$ and $(3,3,0,0)$.
\end{prop}
\begin{figure} [h!]
\begin{center}
\begin{picture}(200,150)
\put(-54,-5){\includegraphics[width=0.80\textwidth]{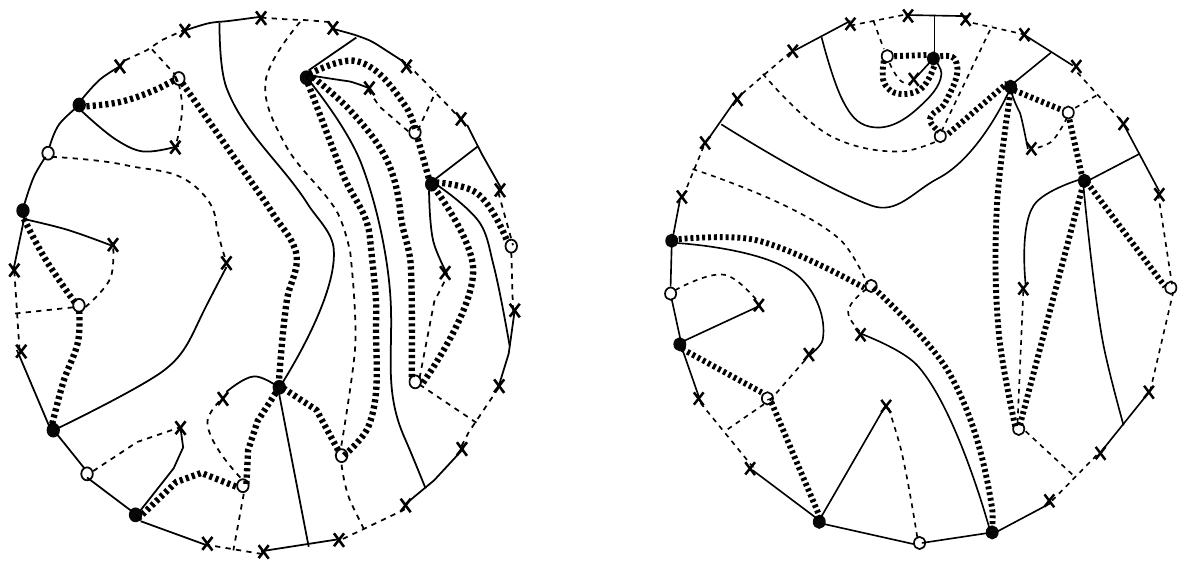}} 
\put(10,70){$1$}
\put(10,-20){$a)$}
\put(170,70){$2$}
\put(170,-20){$b)$}
\end{picture}
\end{center}
\caption{Real trigonal graph of degree $5$ and type I.}
\label{fig: 186_a=3_b=3_and_a=1_b=5}
\end{figure} 
\begin{proof}
Thanks to Theorems \ref{thm: existence_trigonal}, \ref{thm: existence_trigonal_type_I},  if the real graphs associated to $\eta_{a,b,c,c'}$ are completable in degree $5$ to a real trigonal graph of type I (resp. II), then there exist real algebraic trigonal curves of type I (resp. type II) realizing $\eta_{a,b,c,c'}$.\\ 
We can glue $5$ cubic trigonal graphs (see Section \ref{subsubsec: gluing cubics}) in such a way that 
we obtain type I (resp. type II) real trigonal graphs of degree $5$, which complete the real graph associated to $\eta_{a,b,c,c'}$ where $a,b, c$ and $c'$ are such that $c+c'= h, a = j$ and $ b = t$ (resp. $0\leq c+c'< h, 0\leq a < j$ and $ 0\leq b < t$), for $h,j$ and $t$ as listed in $1.-3.$ above. Finally, a type I completion (see Remark \ref{rem: labeling_just_one_number}) of the real graph associated to $\eta_{a,b,c,c'}$ for $(a,b,c,c')=(3,3,0,0)$, resp.$(a,b,c,c') =(1,5,0,0)$, is pictured in Fig. \ref{fig: 186_a=3_b=3_and_a=1_b=5} $a)$, resp. $b)$.
\end{proof}
\begin{figure}[h!]
\begin{center}
\begin{picture}(200,60)
\put(230,-15){\textcolor{black}{$\mathbb{R}\Sigma_6$}}
\put(180,-17){\textcolor{black}{$b)$}}
\put(70,-15){\textcolor{black}{$\mathbb{R}\Sigma_6$}}
\put(15,-17){\textcolor{black}{$a)$}}
\put(-52,-40){\includegraphics[width=0.80\textwidth]{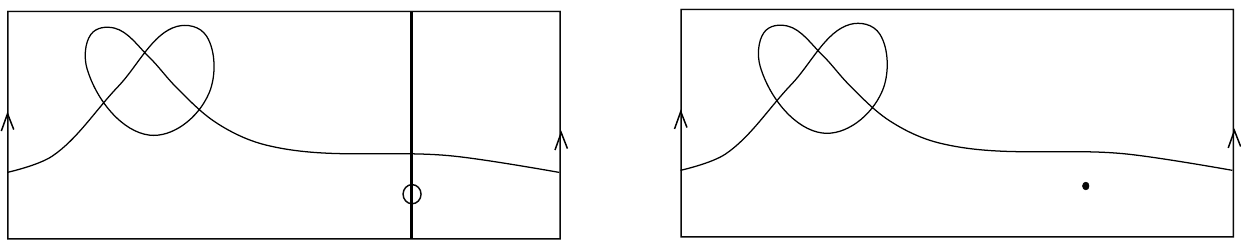}} 
\put(8,45){$a$}
\put(53,45){$b$}
\put(28,10){$c$}
\put(66,10){$c'$}

\put(168,45){$a$}
\put(213,45){$b$}
\put(188,10){$c$}
\put(226,10){$c'$}
\end{picture}
\end{center}
\caption{$a)$ The union of a trigonal $\mathcal{L}$-scheme and a fiber of $\mathcal{L}$ on $\mathbb{R}\Sigma_6$. $b)$ A nodal trigonal $\mathcal{L}$-scheme on $\mathbb{R}\Sigma_6$: $\tilde{\eta}_{a,b,c,c'}$.}
\label{fig: after_snodato_186_gluing_sigma_6}
\end{figure}
\begin{figure} [h!]
\begin{center}
\begin{picture}(200,66)
\put(-14,-5){\includegraphics[width=0.60\textwidth]{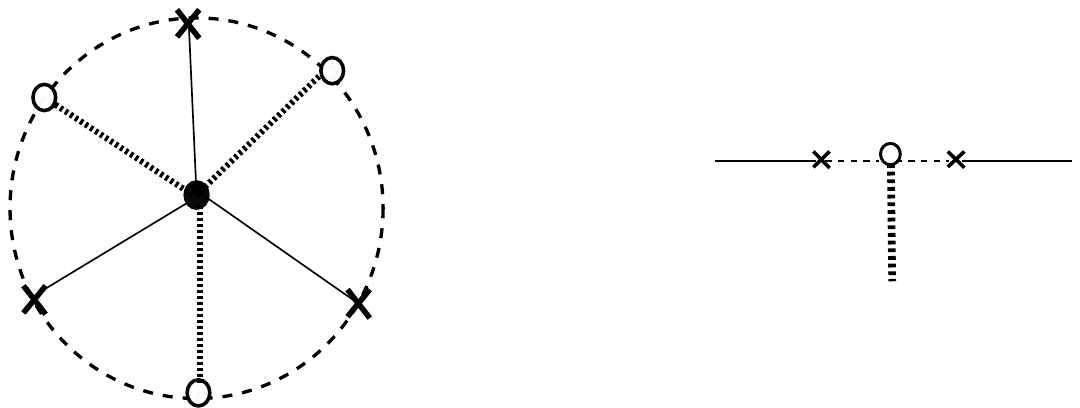}} 
\put(20,5){$3$}
\put(45,5){$3$}
\put(20,45){$1$}
\put(40,45){$1$}
\put(10,25){$2$}
\put(55,25){$2$}
\put(35,-15){$a)$}
\put(163,-15){$b)$}
\put(152,19){$3$}
\put(167,19){$3$}
\end{picture}
\end{center}
\caption{$a)$ Real trigonal graph of degree $1$ and type I: $\xi$ . $b)$ Local type I labeling.}
\label{fig: nodo_sigma_5}
\end{figure} 
\begin{prop}
\label{prop: da sigma_5 a sigma_6}
Let $\tilde{\eta}_{a,b,c,c'}$ be, up to isotopy of $\mathbb{R}\Sigma_6$, a trigonal $\mathcal{L}$-scheme on $\mathbb{R}\Sigma_6$ as depicted in Fig. \ref{fig: after_snodato_186_gluing_sigma_6} $b)$, where letters $a,b,c,c'$ denote numbers of ovals. Then, there exist real algebraic trigonal curves in $\Sigma_6$ realizing the real schemes $\tilde{\eta}_{a,b,c,c'}$ for all $a,b, c,c'$ as listed in Proposition \ref{prop: esistenza_in_sigma_5}.
\end{prop}
\begin{proof}
Thanks to Theorems \ref{thm: existence_trigonal}, \ref{thm: existence_trigonal_type_I},  if the real graphs associated to $\tilde{\eta}_{a,b,c,c'}$ are completable in degree $6$  to a real trigonal graph of type I (resp. II), then there exist real algebraic trigonal curves of type I (resp. type II) realizing $\tilde{\eta}_{a,b,c,c'}$.\\
For $a,b, c,c'$ as listed in Proposition \ref{prop: esistenza_in_sigma_5}, the existence of real trigonal graphs of degree $6$ and type I (resp. type II) completing the real graph associated to $\tilde{\eta}_{a,b,c,c'}$, is equivalent to the existence of those of type I (resp. type II) associated to the $\mathcal{L}$-scheme depicted in Fig. \ref{fig: after_snodato_186_gluing_sigma_6} $a)$, see \cite{Orev03}.\\
Let $\xi$ be the cubic trigonal graph of type I pictured in Fig. \ref{fig: nodo_sigma_5} $a)$. Take any real trigonal graph $\Gamma$ of degree $5$ constructed in the proof of Proposition \ref{prop: esistenza_in_sigma_5}, realizing a trigonal $\mathcal{L}$-scheme $\eta_{a,b,c,c'}$. In a neighborhood of $\Gamma \cap \mathbb{R}P^1$, let us denote by $\delta$ the sub-graph of $\Gamma$ which is as depicted in Fig. \ref{fig: nodo_sigma_5} $b)$ and whose associated $\mathcal{L}$-scheme is the part of $\eta_{a,b,c,c' }$ for which passes one fixed fiber of $\mathcal{L}$ (see Fig. \ref{fig: after_snodato_186_sigma_5}). 
Glue $\Gamma$ along the $\circ$-vertex of $\delta$ to the $\circ$-vertex, with same labeling if $\Gamma$ is of type I, of the cubic trigonal graph $\xi$. The gluing is a real trigonal graph of degree $6$ which completes the real graph associated to the union of a trigonal $\mathcal{L}$-scheme with one fiber of $\mathcal{L}$ as depicted in Fig. \ref{fig: after_snodato_186_gluing_sigma_6} $a)$ on $\mathbb{R}\Sigma_6$, for all $a,b, c,c'$ as in Proposition \ref{prop: esistenza_in_sigma_5}. Besides, the gluing is of type I (resp. of type II) for all $a,b,c,c'$ for which $\Gamma$ is of type I (resp. type II).
\end{proof}
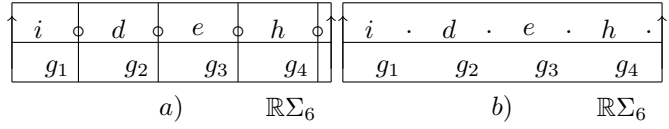
\begin{figure} [h!]
\begin{center}
\begin{tikzpicture} [scale=0.35]
\draw (0.5,0) -- (12.5,0);
\draw [->] (0.5,-1) -- (0.5,1);
\draw (0.5,-1.5) -- (0.5,1.5);
\draw [->] (12.5,-1) -- (12.5,1);
\draw (12.5,-1.5) -- (12.5,1.5);
\draw (0.5,1.5) -- (12.5,1.5);
\draw (0.5,-1.5) -- (12.5,-1.5);
\node at (1.5,0.5) {$i$};
\draw (3,-1.5) -- (3,1.5);
\node at (3,0.4) {$\circ$};
\node at (4.5,0.5) {$d$};
\draw (6,-1.5) -- (6,1.5);
\node at (6,0.4) {$\circ$};
\node at (7.5,0.5) {$e$};
\draw (9,-1.5) -- (9,1.5);
\node at (9,0.4) {$\circ$};
\node at (10.5,0.5) {$h$};
\draw (12,-1.5) -- (12,1.5);
\node at (12,0.4) {$\circ$};
\node at (2.2,-1) {$g_1$};
\node at (5.2,-1) {$g_2$};
\node at (8.2,-1) {$g_3$};
\node at (11.2,-1) {$g_4$};
\node at (11,-2.5) {$\mathbb{R}\Sigma_6$};
\node at (6.5,-2.5) {$a)$};
\end{tikzpicture}
\begin{tikzpicture} [scale=0.35]
\draw (0.5,0) -- (12.5,0);
\draw [->] (0.5,-1) -- (0.5,1);
\draw (0.5,-1.5) -- (0.5,1.5);
\draw [->] (12.5,-1) -- (12.5,1);
\draw (12.5,-1.5) -- (12.5,1.5);
\draw (0.5,1.5) -- (12.5,1.5);
\draw (0.5,-1.5) -- (12.5,-1.5);
\node at (1.5,0.5) {$i$};
\node at (3,0.4) {$.$};
\node at (4.5,0.5) {$d$};
\node at (6,0.4) {$.$};
\node at (7.5,0.5) {$e$};
\node at (9,0.4) {$.$};
\node at (10.5,0.5) {$h$};
\node at (12,0.4) {$.$};
\node at (2.2,-1) {$g_1$};
\node at (5.2,-1) {$g_2$};
\node at (8.2,-1) {$g_3$};
\node at (11.2,-1) {$g_4$};
\node at (11,-2.5) {$\mathbb{R}\Sigma_6$};
\node at (6.5,-2.5) {$b)$};
\end{tikzpicture}
\end{center}
\caption{$a)$ The union of a trigonal $\mathcal{L}$-scheme and four fibers of $\mathcal{L}$ on $\mathbb{R}\Sigma_6$: $\tilde{\eta}_{i,d,e,h,g}$. $b)$ A nodal trigonal $\mathcal{L}$-scheme on $\mathbb{R}\Sigma_6$: $\eta_{i,d,e,h,g}$.}
\label{fig: secondo_disegno_con_ovali}
\end{figure} 
\begin{prop}
\label{prop: secondo_disegno}
Let $\eta_{i,d,e,h,g}$ be, up to isotopy of $\mathbb{R}\Sigma_6$, the trigonal $\mathcal{L}$-scheme on $\mathbb{R}\Sigma_6$ depicted in Fig. \ref{fig: secondo_disegno_con_ovali} $b)$, where letters $i,d,e,h,g_j$, for $j=1,2,3,4$, denote numbers of ovals. Let $g$ be $\sum_{j=1}^{4}g_j$ and let $s,k$ be non-negative integer numbers. Then, there exist real algebraic trigonal curves in $\Sigma_6$ realizing the real schemes $\eta_{i,d,e,h,g}$ for all $i,d,e,h,g$ such that $0 \leq g \leq s, 0 \leq i+d+e+h \leq k$, where $s,k$ are as follows:
\begin{enumerate}
\item $s+k=12$ with
$s \in \{6,10\}$, 
\item $s+k=10$ with
$s \in \{5,9\}$,
\item $s+k=8$ with
$s \in \{0,4,6,8\}$,
\end{enumerate}
Furthermore, the real trigonal curves with $g= s$ and $ i+d+e+h = k$, are of type I. Also, there exist real algebraic trigonal curves of type I in $\Sigma_6$ realizing $\eta_{i,d,e,h,g}$ for 
\begin{enumerate}
\setcounter{enumi}{3}
\item $i+d+e+h+g=8$ with
$g=0$;
\item $i+d+e+h+g=6$ with
$g \in \{1,3,5\}$,
\item  $i+d+e+h+g=4$ with
$g \in \{2,4\}$.
\end{enumerate}
\end{prop}
\begin{proof}
Thanks to Theorems \ref{thm: existence_trigonal}, \ref{thm: existence_trigonal_type_I},  if the real graphs associated to $\eta_{i,d,e,h,g}$ are completable in degree $6$  to a real trigonal graph of type I (resp. II), then there exist real algebraic trigonal curves of type I (resp. type II) realizing $\eta_{i,d,e,h,g}$.\\ 
Let $\tilde{\eta}_{i,d,e,h,g}$ be, up to isotopy of $\mathbb{R}\Sigma_6$, the union of a trigonal $\mathcal{L}$-scheme with four fibers of $\mathcal{L}$ on $\mathbb{R}\Sigma_6$ as depicted in Fig. \ref{fig: secondo_disegno_con_ovali} $a)$. Remark that, for $i,d,e,h,g$ as listed in $1. - 6.$ above, the existence of real trigonal graphs of degree $6$ and type I (resp. type II) completing the real graph associated to $\eta_{i,d,e,h,g}$ is equivalent to the existence of those of type I (resp. type II) associated to $\tilde{\eta}_{i,d,e,h,g}$ (see \cite{Orev03}). \\
We can glue $6$ cubic trigonal graphs in such a way that we obtain real trigonal graphs of degree $6$, which complete the real graph associated to $\tilde{\eta}_{i,d,e,h,g}$ where $i,d,e,h,g$ are such that $0\leq g \leq s$ and $ 0\leq i+d+e+h \leq k$, for $s,k$ as listed in $1.-3.$ above. Type I completions of the real graph associated to $\tilde{\eta}_{i,d,e,h,g}$, for values listed in $4.-6.$ above, are 
pictured in Fig. \ref{fig: compl_grafo_12/0_9/1_6/2_7/3_5/5_4/4_type_I}.
\end{proof}

\begin{figure}[h!]
\begin{center}
\begin{picture}(200,230)
\put(-75,-20){\includegraphics[width=0.9\textwidth]{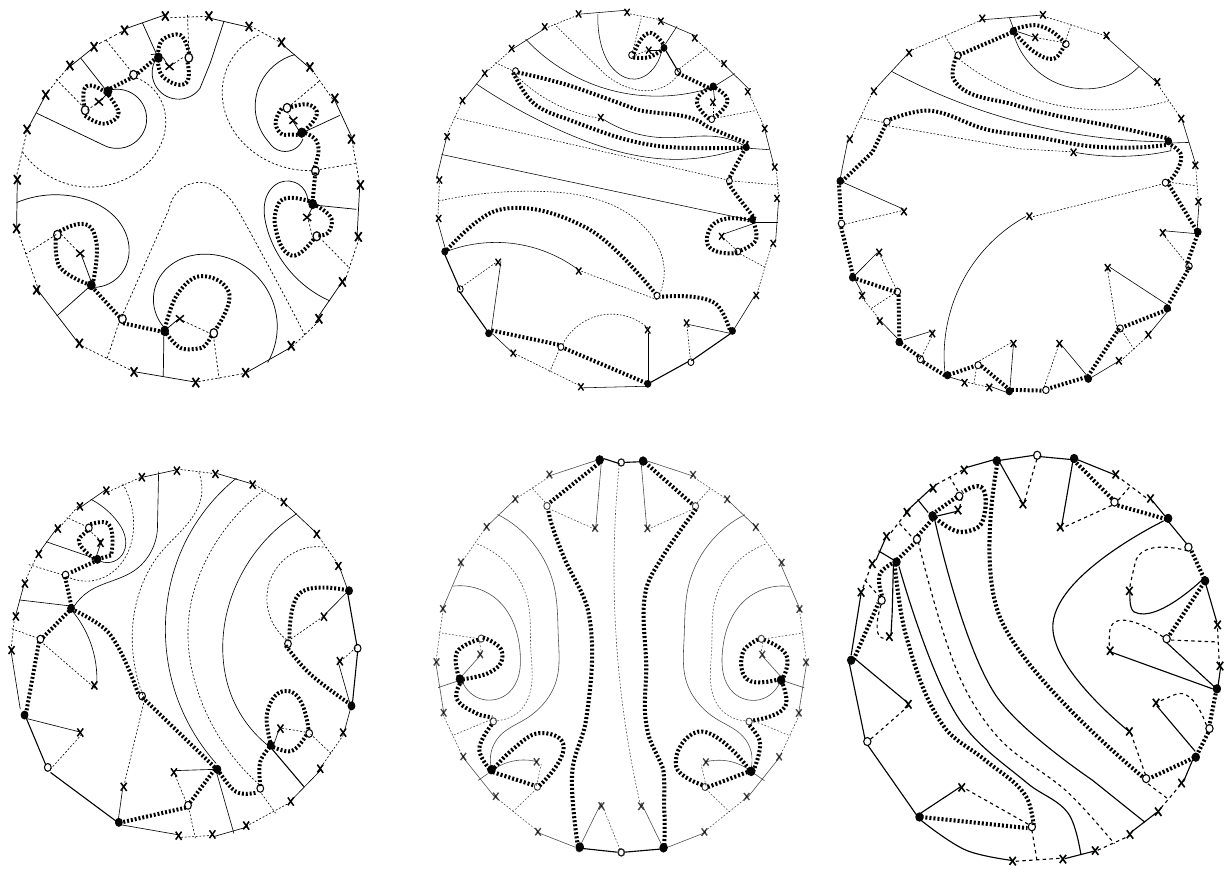}}
\put(-25,200){$3$}
\put(-26,130){$a)$}
\put(94,195){$2$}
\put(90,130){$b)$}
\put(202,190){$1$}
\put(202,130){$c)$} 
\put(-33,70){$2$}
\put(-26,-5){$d)$}
\put(94,70){$1$}
\put(90,-5){$e)$}
\put(203,70){$2$}
\put(202,-5){$f)$} 
\end{picture}
\end{center}
\caption{Real trigonal graphs of degree $6$ and type I.}
\label{fig: compl_grafo_12/0_9/1_6/2_7/3_5/5_4/4_type_I}
\end{figure} 
\begin{figure}[h!]
\begin{center}
\begin{picture}(200,180)
\put(-55,5){\includegraphics[width=0.8\textwidth]{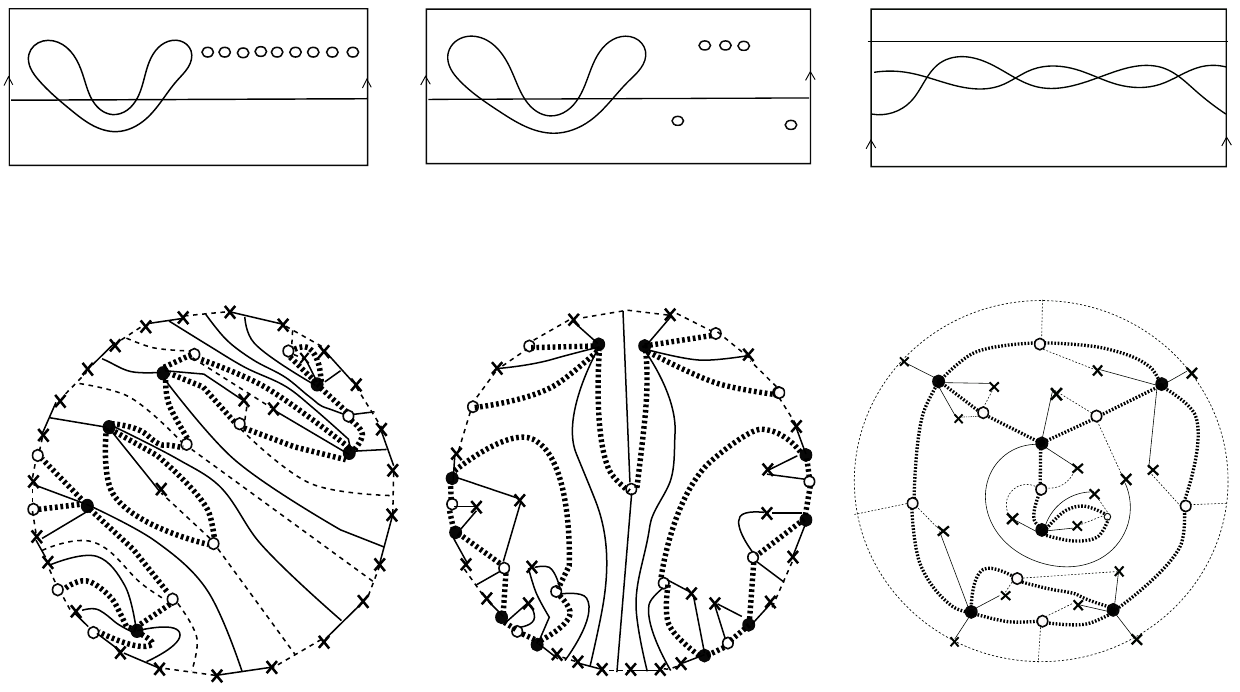}} 

\put(190,0){\textcolor{black}{$f)$}}
\put(92,0){\textcolor{black}{$e)$}}
\put(90,55){\textcolor{black}{$3$}}
\put(-10,0){\textcolor{black}{$d)$}}
\put(0,30){$2$}
\put(190,130){\textcolor{black}{$c)$}}
\put(215,130){$\mathbb{R}\Sigma_6$}
\put(90,130){\textcolor{black}{$b)$}}
\put(115,130){$\mathbb{R}\Sigma_6$}

\put(-15,130){\textcolor{black}{$a)$}}
\put(10,130){$\mathbb{R}\Sigma_6$}

\end{picture}
\end{center}
\caption{Trigonal $\mathcal{L}$-schemes on $\mathbb{R}\Sigma_6$ and the completion of their real graphs in degree 6.}
\label{fig: primo_dessin_enfant_a=9_a=3_b=2_nido_max_compl}
\end{figure} 
\begin{prop}
\label{prop: hyperbolic_e_non hyperbolic}
There exist real algebraic trigonal curves of type I in $\Sigma_6$ realizing the trigonal $\mathcal{L}$-schemes respectively depicted in Fig. \ref{fig: primo_dessin_enfant_a=9_a=3_b=2_nido_max_compl} $a)$ and $b)$. Moreover, there exists a real algebraic trigonal curve in $\Sigma_6$ realizing the hyperbolic (see Remark \ref{rem: hyperbolic_type I}) $\mathcal{L}$-scheme depicted in Fig. \ref{fig: primo_dessin_enfant_a=9_a=3_b=2_nido_max_compl} $c)$.
\end{prop}
\begin{proof}
Thanks to Theorems \ref{thm: existence_trigonal}, \ref{thm: existence_trigonal_type_I},  if the real graphs associated to the real schemes in the statement are completable in degree $6$ to a real trigonal graph of type I (resp. II), then there exist real algebraic trigonal curves of type I (resp. type II) realizing them.\\
Respective completions in degree $6$ of the real graphs associated to the $\mathcal{L}$-schemes in Fig. \ref{fig: primo_dessin_enfant_a=9_a=3_b=2_nido_max_compl} $a),b)$ and $c)$ are pictured in Fig. \ref{fig: primo_dessin_enfant_a=9_a=3_b=2_nido_max_compl} $d)$, $e)$ and $f)$. Furthermore, the trigonal graphs depicted in Fig. \ref{fig: primo_dessin_enfant_a=9_a=3_b=2_nido_max_compl} $d)$ and $e)$ are of type I.
\end{proof}

\begin{figure}[h!]
\begin{center}
\begin{picture}(100,90)
\put(-175,-30){\includegraphics[width=1.2\textwidth]{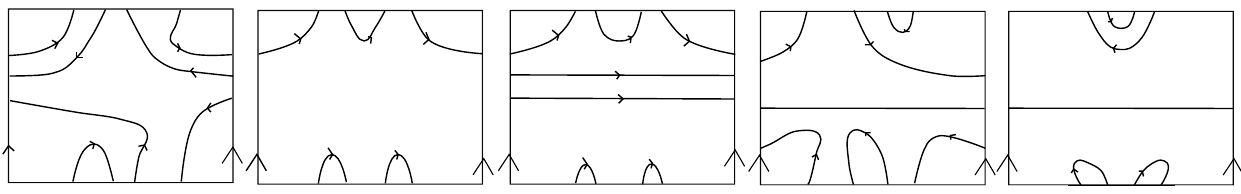}} 

\put(-135,-19){$a)$}
\put(-137,8){$a$}
\put(-135,38){$c$}
\put(-123,48){$b$}
\put(-118,18){$c'$}

\put(-47,-19){$b)$}
\put(-80,47){$i$}
\put(-77,25){$g_1$}
\put(-60,-6){$d$}
\put(-57,25){$g_2$}
\put(-46,50){$e$}
\put(-45,25){$g_3$}
\put(-34,-6){$h$}
\put(-32,25){$g_4$}

\put(134,-19){$d)$}
\put(119,0){\textcolor{black}{$a'$}}
\put(120,40){\textcolor{black}{$b'$}}
\put(42,-19){$c)$}
\put(193,0){\textcolor{black}{$a'$}}
\put(195,40){\textcolor{black}{$b'$}}
\put(224,-19){$e)$}

\end{picture}
\end{center}
\caption{$\mathcal{L}$-schemes $\eta_1,\eta_2,\eta_3,\eta_4,\eta_5$ on $\mathbb{R}\Sigma_2$.}
\label{fig: c_i}
\end{figure}
\begin{prop}
\label{prop: costruzione_più_cicciona}
Let $\eta_1$, resp. $\eta_2, \eta_3$ and $\eta_4$, be a trigonal $\mathcal{L}$-scheme on $\mathbb{R}\Sigma_2$, up to isotopy of $\mathbb{R}\Sigma_2$, as depicted in Fig. \ref{fig: c_i} $a)$, resp. $b),c)$ and $d)$, where $a,b,c,c'$, resp. $i,d,e,h,g$, resp. $a',b'$ denote numbers of ovals. Then, such $\mathcal{L}$-scheme is realizable by a non-singular real algebraic curve $C_1$ (resp. $C_2,C_3$ and $C_4$) of bidegree $(3,4)$ on $\Sigma_2$ for $a,b, c,c'$ as listed in Proposition \ref{prop: esistenza_in_sigma_5} (resp. $i,d,e,h,g$ as listed in Proposition \ref{prop: secondo_disegno}, resp. $(a',b')=(3,2)$ and $(9,0)$).
\end{prop}
\begin{figure}[h!]
\begin{center}
\begin{picture}(100,70)
\put(100,-10){\textcolor{black}{$\mathbb{R}\Sigma_5$}}
\put(20,-10){\textcolor{black}{$\mathbb{R}\Sigma_4$}}
\put(-60,-10){\textcolor{black}{$\mathbb{R}\Sigma_3$}}
\put(190,-10){\textcolor{black}{$\mathbb{R}\Sigma_6$}}
\put(-125,-30){\includegraphics[width=0.95\textwidth]{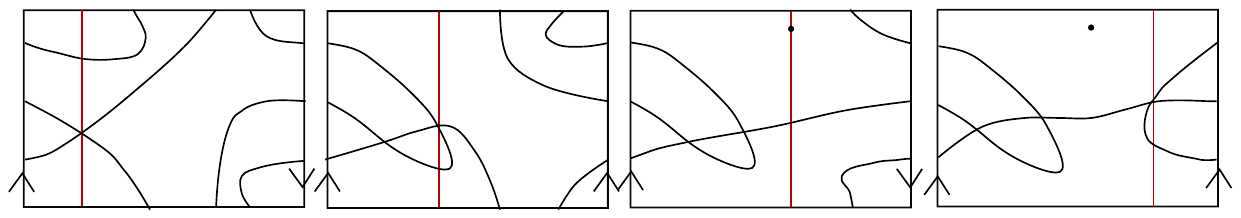}} 

\put(-97,23){$p_1$}
\put(-63,5){$b$}
\put(-80,50){$a$}
\put(-75,5){$c$}
\put(-62,50){$c'$}

\put(-11,28){$p_2$}
\put(26,5){$c'$}
\put(5,50){$c$}
\put(5,5){$a$}
\put(22,50){$b$}
\put(89,52){$p_3$}
\put(106,5){$c'$}
\put(90,41){$c$}
\put(90,5){$a$}
\put(102,41){$b$}
\put(202,26){$p_4$}
\put(193,5){$b$}
\put(175,50){$c$}
\put(180,5){$a$}
\put(189,50){$c'$}
\end{picture}
\end{center}
\caption{Birational transformation of the pair $(\mathbb{R}\Sigma_6, \mathbb{R}\tilde{C_1})$, from right to left.}
\label{fig: blowup_down C_1}
\end{figure}
\begin{figure}[h!]
\begin{center}
\begin{picture}(100,80)
	
\put(100,-10){\textcolor{black}{$\mathbb{R}\Sigma_5$}}
\put(20,-10){\textcolor{black}{$\mathbb{R}\Sigma_4$}}
\put(-60,-10){\textcolor{black}{$\mathbb{R}\Sigma_3$}}
\put(190,-10){\textcolor{black}{$\mathbb{R}\Sigma_6$}}
\put(-125,-30){\includegraphics[width=0.95\textwidth]{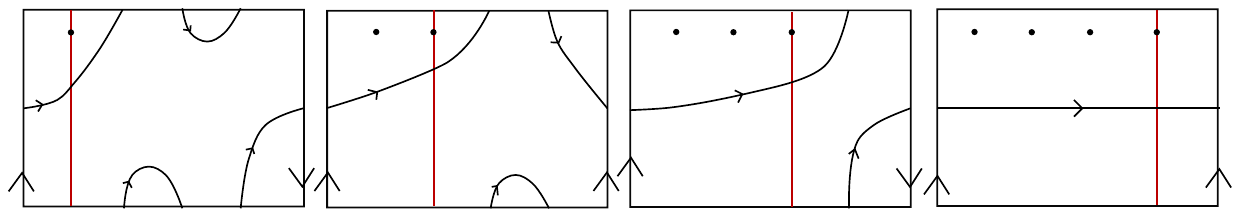}} 

\put(-114,53){$p_1$}
\put(-115,40){$i$}
\put(-100,50){$d$}
\put(-85,5){$e$}
\put(-67,53){$h$}

\put(-113,25){$g_1$}
\put(-98,25){$g_2$}
\put(-83,25){$g_3$}
\put(-65,25){$g_4$}
\put(-12,54){$p_2$}
\put(-26,50){$i$}
\put(-13,43){$d$}
\put(3,50){$e$}
\put(21,5){$h$}

\put(-24,25){$g_1$}
\put(-13,25){$g_2$}
\put(5,25){$g_3$}
\put(23,25){$g_4$}
\put(89,55){$p_3$}
\put(60,50){$i$}
\put(75,50){$d$}
\put(90,43){$e$}
\put(105,50){$h$}

\put(62,25){$g_1$}
\put(77,25){$g_2$}
\put(92,25){$g_3$}
\put(107,25){$g_4$}

\put(205,55){$p_4$}
\put(145,50){$i$}
\put(160,50){$d$}
\put(175,50){$e$}
\put(190,50){$h$}

\put(147,25){$g_1$}
\put(162,25){$g_2$}
\put(177,25){$g_3$}
\put(192,25){$g_4$}
\end{picture}
\end{center}
\caption{Birational transformation of the pair $(\mathbb{R}\Sigma_6, \mathbb{R}\tilde{C_2})$, from right to left.}
\label{fig: blowup_down C_2}
\end{figure}
\begin{figure}[h!]
\begin{center}
\begin{picture}(100,80)
	
\put(100,-10){\textcolor{black}{$\mathbb{R}\Sigma_5$}}
\put(20,-10){\textcolor{black}{$\mathbb{R}\Sigma_4$}}
\put(-60,-10){\textcolor{black}{$\mathbb{R}\Sigma_3$}}
\put(190,-10){\textcolor{black}{$\mathbb{R}\Sigma_6$}}
\put(-125,-30){\includegraphics[width=0.95\textwidth]{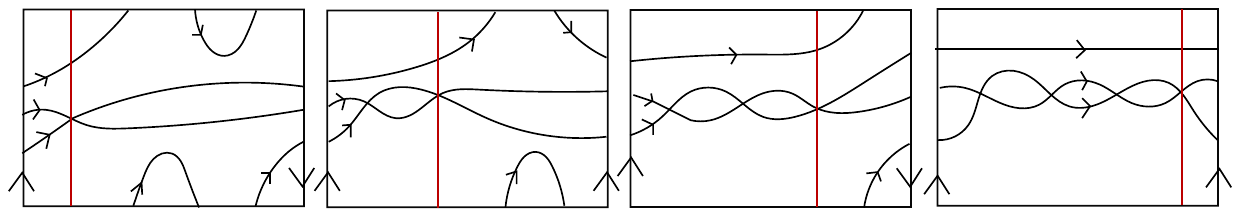}} 
\put(-103,27){$p_1$}

\put(3,30){$p_2$}

\put(110,27){$p_3$}
\put(200,30){$p_4$}
\end{picture}
\end{center}
\caption{Birational transformation of the pair $(\mathbb{R}\Sigma_6, \mathbb{R}\tilde{C_3})$, from right to left.}
\label{fig: blowup_down C_3}
\end{figure}
\begin{figure}[h!]
\begin{center}
\begin{picture}(100,80)
\put(105,-15){\textcolor{black}{$\mathbb{R}\Sigma_5$}}
\put(20,-15){\textcolor{black}{$\mathbb{R}\Sigma_4$}}
\put(-60,-15){\textcolor{black}{$\mathbb{R}\Sigma_3$}}
\put(195,-15){\textcolor{black}{$\mathbb{R}\Sigma_6$}}
\put(-125,-30){\includegraphics[width=0.95\textwidth]{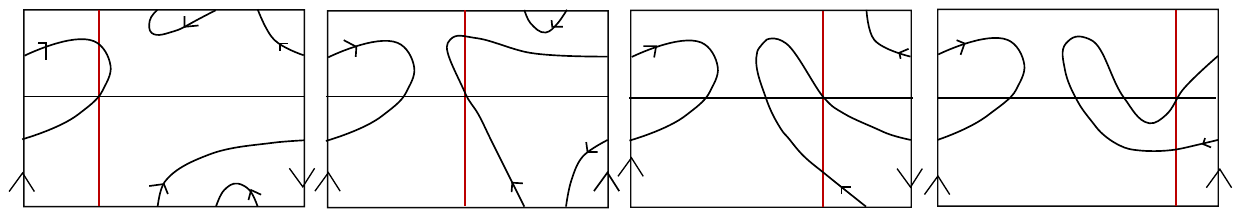}} 
\put(-105,32){$p_1$}
\put(-90,40){$a'$}
\put(-90,5){$b'$}

\put(11,32){$p_2$}
\put(-5,40){$a'$}
\put(-5,5){$b'$}

\put(110,32){$p_3$}
\put(80,40){$a'$}
\put(80,5){$b'$}

\put(200,32){$p_4$}
\put(167,40){$a'$}
\put(167,5){$b'$}
\end{picture}
\end{center}
\caption{Birational transformation of the pair $(\mathbb{R}\Sigma_6, \mathbb{R}\tilde{C_4})$, from right to left.}
\label{fig: blowup_down C_4}
\end{figure}
\begin{figure}[h!]
\begin{center}
\begin{picture}(100,115)
	
\put(-185,-90){\includegraphics[width=1.3\textwidth]{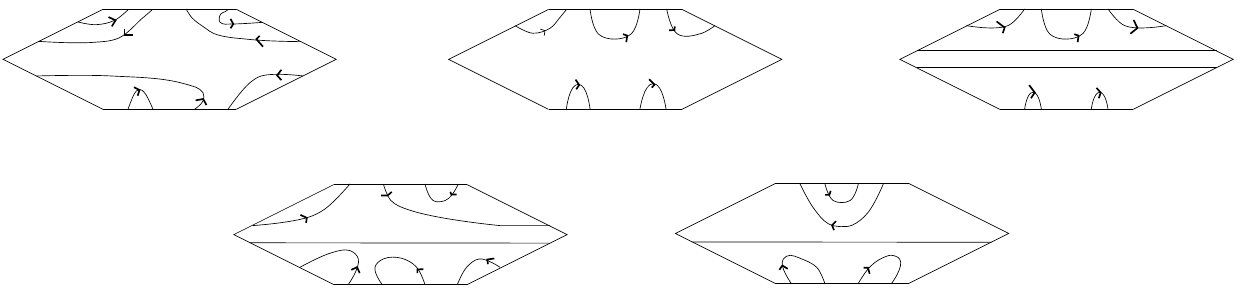}} 

\put(-120,63){$a)C_1$}
\put(-118,78){$a$}
\put(-125,103){$c$}
\put(-106,106){$b$}
\put(-100,83){$c'$}

\put(50,63){$b)C_2$}
\put(24,106){$i$}
\put(27,89){$g_1$}
\put(36,75){$d$}
\put(38,89){$g_2$}
\put(47,108){$e$}
\put(54,89){$g_3$}
\put(65,75){$h$}
\put(67,89){$g_4$}

\put(228,63){$c)C_3$}
\put(-33,-12){$d)C_4$}
\put(-45,15){\textcolor{black}{$a'$}}
\put(-43,30){\textcolor{black}{$b'$}}
\put(139,-12){$e)C_5$}
\put(103,15){\textcolor{black}{$a'$}}
\put(108,30){\textcolor{black}{$b'$}}
\end{picture}
\end{center}
\caption{Charts of real curves $C_1,C_2,C_3,C_4,C_5$ of bidegree $(3,4)$ and type I in $\Sigma_2$.}
\label{fig: charts_c_i}
\end{figure}
\begin{proof}
Let us denote by $\tilde{C}_1$ (resp. $\tilde{C}_2$, $\tilde{C}_3$ and $\tilde{C}_4$) any real algebraic trigonal curves in $\Sigma_6$ constructed in Propositions \ref{prop: da sigma_5 a sigma_6} (resp. Proposition \ref{prop: secondo_disegno}, resp. Proposition \ref{prop: hyperbolic_e_non hyperbolic}). Let us consider, as defined in Section \ref{subsec: Hirzebruch_surf}, for each curve $\tilde{C}_j$ the birational transformation $\Xi_j=:
\beta^{-1}_{p_{1}}\beta^{-1}_{p_{2}}\beta^{-1}_{p_{3}}\beta^{-1}_{p_{4}}: (\Sigma_6,
\tilde{C}_j)\rightarrow (\Sigma_2, C_j)$, where the points $p_k$, $k=1,2,3,4$, are real double points such that $p_4$ belong to $\mathbb{R}\tilde{C}_j$ and $p_k$, $k=3,2,1$, to the image of $\mathbb{R}\tilde{C}_j$ via $\beta^{-1}_{p_{k+1}}$. In Fig. \ref{fig: blowup_down C_1}, \ref{fig: blowup_down C_2}, \ref{fig: blowup_down C_3} and \ref{fig: blowup_down C_4} we depict in red the fiber of $\mathbb{R}\Sigma_{k+2}$ intersecting the point $p_k$.\\
The birational transformation $\beta^{-1}_{p_{2}}\beta^{-1}_{p_{3}}\beta^{-1}_{p_{4}}((\mathbb{R}\Sigma_6, \mathbb{R}\tilde{C}_j))$ is depicted in Fig. \ref{fig: blowup_down C_1}, resp. in Fig. \ref{fig: blowup_down C_2}, Fig. \ref{fig: blowup_down C_3}, Fig. \ref{fig: blowup_down C_4} from right to left and $\Xi_j((\mathbb{R}\Sigma_6, \mathbb{R}\tilde{C}_j))$ in Fig. \ref{fig: c_i} $a)$ for $j=1$, resp. in Fig. \ref{fig: c_i} $b),c)$ and $d)$ for $j=2,3,4$.\\
\end{proof}
Since we will apply Viro's patchworking method in Section \ref{subsec: patch} to construct real algebraic curves of bidegree $(5,0)$ on $\Sigma_2$ (see \cite{Viro84a}, \cite{Viro84b}, \cite{Viro89}, \cite{Viro06}), in Fig. \ref{fig: charts_c_i} $a),b),c)$ and $d)$ we depict the charts of non-singular real algebraic curves of bidigree $(3,4)$ in $\Sigma_2$ constructed in Proposition \ref{prop: costruzione_più_cicciona}. Moreover, performing a coordinates transformation to a curve $C_4$ with chart as depicted in Fig. \ref{fig: charts_c_i} $d)$, we obtain a type I real algebraic curve $C_5$ of bidigree $(3,4)$ in $\Sigma_2$ with chart, resp. real $\mathcal{L}$-scheme, as depicted in Fig. \ref{fig: charts_c_i} $e)$, resp. as depicted in Fig. \ref{fig: c_i} $e)$, where $a',b'$ still denote numbers of ovals and $(a',b')=(3,2)$ or $(9,0)$.
\subsection{Final constructions and patchworking}
\label{subsec: patch}
In this Section we end the proof of Theorems \ref{thm: maximal_thm}, \ref{thm: m-1_curves_thm}, \ref{thm: type_I,II_m-2_thm}, \ref{thm: type_I_II_thm}. We need Viro's patchworking method. Most of all, we use original Viro’s patchworking method which is a tool for constructing non-singular real algebraic hypersurfaces with prescribed topology in real toric varieties (\cite{Viro84a}, \cite{Viro84b}, \cite{Viro89}, \cite{Viro06}). Finally, for a particular construction, we use a variant of the patchworking developed by Shustin (\cite{Shus98}, \cite{Shus02}, \cite{Shus05}, \cite{Shus06}) which exploits the deformation pattern technique and allows to glue charts of polynomials presenting a tangency point with the boundary of the chart.
\begin{rem}
\label{rem: fine}
We want to construct bidegree $(5,5)$ non-singular real algebraic curves on the quadric ellipsoid realizing all real schemes listed in Theorems \ref{thm: maximal_thm}, \ref{thm: m-1_curves_thm}, \ref{thm: type_I,II_m-2_thm}, \ref{thm: type_I_II_thm}. We can reduce the problem of construction of such bidegree $(5,5)$ real curves to the construction of bidegree $(5,0)$ real curves in $\Sigma_2$ with topology prescribed as follows. Let $B$ be one of the real schemes listed in Theorems \ref{thm: maximal_thm}, \ref{thm: m-1_curves_thm}, \ref{thm: type_I,II_m-2_thm}, \ref{thm: type_I_II_thm}. We can choose a real scheme $\eta$ in $\mathbb{R}\Sigma_2$ such that from the arrangement of the triplet $(\mathbb{R}\Sigma_2,\mathbb{R}E_2,\eta)$ one recovers the arrangement $B$ on a $2$-sphere as explained in Section \ref{subsec: the quadratic cone}. Thanks to Proposition \ref{prop: construct_on_torus_to_construct_on_ellipsoid}, if we construct a bidegree $(5,0)$ real curve realizing $\eta $ in $\Sigma_2$, we also construct a bidegree $(5,5)$ real curve realizing $B$ on the quadric ellipsoid. 
\end{rem}
\begin{note}
In Propositions \ref{prop: big_first_gluing}, \ref{prop: big_first_gluing_2}, \ref{prop: big_first_gluing_3} the real schemes marked with the symbol $^{\circ}$ (resp. $^*$) are realized by a real algebraic curve of type I (resp. type II). 
\end{note}
\begin{prop}
\label{prop: big_first_gluing}
All the real schemes in the following list are realizable by non-singular real algebraic curves of bidegree $(5,5)$ on the quadric ellipsoid:
\begin{enumerate}
\item all real schemes listed in Theorem \ref{thm: maximal_thm} but the real schemes $1 \quad \sqcup \quad  \langle 4 \rangle \quad \sqcup \quad \langle 10 \rangle$ and $1 \quad \sqcup \quad  \langle 7 \rangle \quad \sqcup \quad \langle 7 \rangle$;
\item all real schemes listed in Theorem \ref{thm: m-1_curves_thm} but the real schemes $\langle 4 \rangle \quad \sqcup \quad \langle 10 \rangle$ and  $\langle 7 \rangle \quad \sqcup \quad \langle 7 \rangle$;
\item all real schemes listed in Theorem \ref{thm: type_I,II_m-2_thm} but the real scheme $ \langle 4 \rangle \quad \sqcup \quad \langle9 \rangle^{\circ}$; 
\item all real schemes listed in Theorem \ref{thm: type_I_II_thm} but the real schemes $13^{\circ}$, $1 \quad \sqcup \quad \langle 1 \rangle \quad \sqcup \quad \langle 9 \rangle^{\circ}$,  $1 \quad \sqcup \quad \langle 3 \rangle \quad \sqcup \quad  \langle 7 \rangle^{\circ}$, $\langle 1 \rangle \quad \sqcup \quad \langle 4 \rangle^{\circ}$, $1$ and $0$.
\end{enumerate}
\end{prop}
\begin{figure}[h!]
\begin{center}
\begin{picture}(100,55)
\put(-137,-15){\includegraphics[width=1\textwidth]{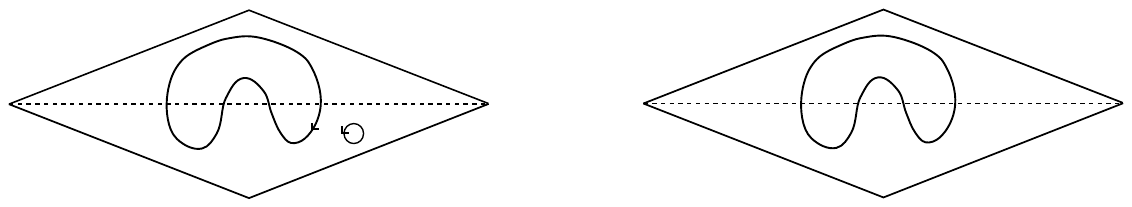}} 
\put(-53,-12){$a)D_1$}
\put(139,-12){$b)D_2$}
\end{picture}
\end{center}
\caption{Charts and arrangements with respect of the coordinate axis $\{y=0\}$ of real algebraic curves of bidegree $(2,0)$ on $\Sigma_2$.}
\label{fig: bid_(2,2)}
\end{figure}
\begin{figure}[t!]
\begin{center}
\begin{picture}(100,500)	
\put(-135,-12){\includegraphics[width=1\textwidth]{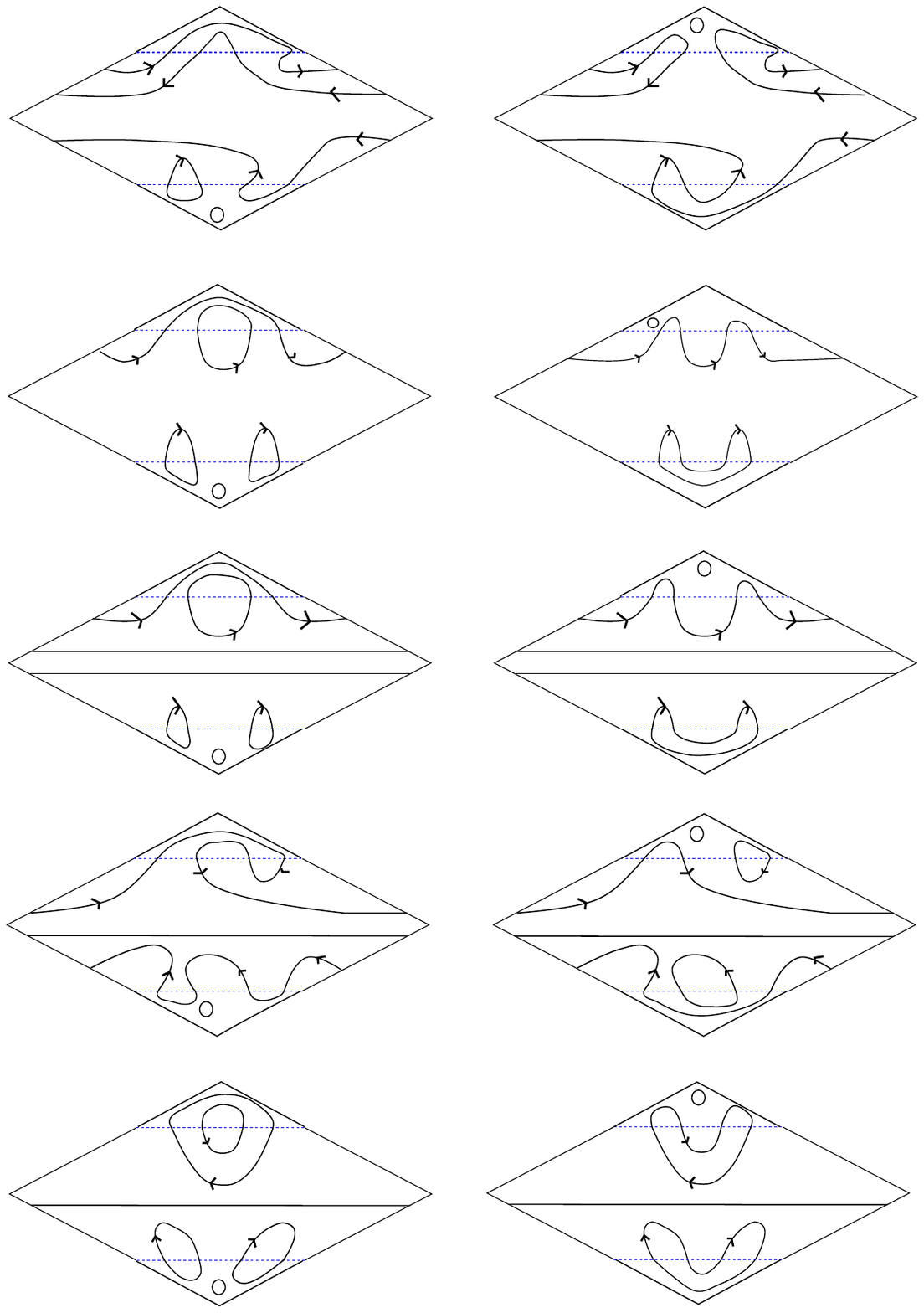}} 
\put(-50,402){$a)$}
\put(-48,438){$a$}
\put(-55,463){$c$}
\put(-30,476){$b$}
\put(-24,453){$c'$}

\put(141,402){$b)$}
\put(142,438){$a$}
\put(135,463){$c$}
\put(163,476){$b$}
\put(163,453){$c'$}

\put(-50,295){$c)$}
\put(-82,368){$i$}
\put(-79,349){$g_1$}
\put(-66,327){$d$}
\put(-66,349){$g_2$}
\put(-49,368){$e$}
\put(-47,349){$g_3$}
\put(-32,327){$h$}
\put(-27,349){$g_4$}

\put(139,295){$d)$}
\put(110,368){$i$}
\put(112,349){$g_1$}
\put(128,327){$d$}
\put(126,349){$g_2$}
\put(141,368){$e$}
\put(143,349){$g_3$}
\put(156,327){$h$}
\put(158,349){$g_4$}

\put(-50,191){$e)$}
\put(139,191){$f)$}
\put(-50,88){$g)$}
\put(-64,130){\textcolor{black}{$a'$}}
\put(-60,140){\textcolor{black}{$b'$}}
\put(139,88){$h)$}
\put(127,130){\textcolor{black}{$a'$}}
\put(131,140){\textcolor{black}{$b'$}}
\put(-50,-18){$i)$}
\put(-95,20){\textcolor{black}{$a'$}}
\put(-93,40){\textcolor{black}{$b'$}}
\put(139,-18){$l)$}
\put(101,22){\textcolor{black}{$a'$}}
\put(105,50){\textcolor{black}{$b'$}}
\end{picture}
\end{center}
\caption{Charts of curves of bidegree $(5,0)$ in $\Sigma_2$ obtained patchworking the charts of the real algebraic curves $C_i$ from Proposition \ref{prop: costruzione_più_cicciona} and $D_1$.}
\label{fig: patchwork_c_i}
\end{figure}
\begin{proof}
Thanks to Remark \ref{rem: fine}, to prove the statement we just need to construct bidegree $(5,0)$ real curves in $\Sigma_2$ whose charts are the patchworking of one of the charts depicted in Fig. \ref{fig: charts_c_i}, with one of those depicted in Fig. \ref{fig: bid_(2,2)}. First of all, let us construct real algebraic curves $D_1, D_2$ of bidegree $(2,0)$ in $\Sigma_2$ intersecting a given real curve of bidegree $(1,0)$ in four fixed real points and with charts respectively as depicted in $a)$ and $b)$ of Fig. \ref{fig: bid_(2,2)}. Let $\tilde{H}$ be a bidegree $(1,0)$ real algebraic curve in $\Sigma_2$. For any two fixed distinct real points on $\tilde{H}$, there exists a bidegree $(1,0)$ real algebraic curve $H$ passing through them. Let $P_0(x,y)P_1(x,y)=0$ be a polynomial equation defining the union of $\tilde{H}$ and $H$. For any four fixed distinct real points on a connected component $\mathcal{F}$ of $\mathbb{R}\tilde{H}\setminus \mathbb{R}H$, there exist two real curves $H_1$ and $H_2$ of bidegree $(1,0)$, such that $H_1 \cup H_2$ passes through the fixed four points. Replace the left side of the equation $P_0(x,y)P_1(x,y)=0$ with $P_0(x,y)P_1(x,y)+ \varepsilon f_1(x,y)f_2(x,y)$, where $f_i(x,y)=0$ is an equation for $H_i$ and $\varepsilon$ is a sufficient small real number. Up to a choice of the sign of $\varepsilon$, one constructs a small perturbation $D_1$ of $\tilde{H} \cup H$, where $D_1$ is a bidegree $(2,0)$ non-singular real curve such that $\bigcup_{i=1}^2H_i \cap \tilde{H}= D_1 \cap \tilde{H} $ and whose chart is as depicted in $a)$ of Fig. \ref{fig: bid_(2,2)}, where $\tilde{H}$, up to a change of coordinates, is $\{y=0\}$ (in dashed). Analogously, it is easy to construct a bidegree $(2,0)$ real curve $D_2$ whose chart, for any four fixed real points on the coordinates axis $\{y=0\}$, is as depicted in $b)$ of Fig. \ref{fig: bid_(2,2)} and intersect $\{y=0\}$ in the four fixed points. Thanks to Viro's patchworking method, we realize every real scheme listed in $1.-4.$ gluing the polynomial and the chart of a real algebraic curve $C_i$, $i=1,2,3,4,5$, constructed in Proposition $4.5$ and at the end of Section $4.1$, with those of a real algebraic curve $D_j$, $j=1,2$. In the case we patchwork real curves of type I, we get type I bidegree $(5,0)$ real algebraic curves on $\Sigma_2$ only if, fixed a complex orientation on the $C_i$'s, up to a choice of a complex orientation on $D_1$, we can patchwork such curves in a compatible way with their complex orientations. For example, let us consider the curves $C_i$ of type I with charts as depicted in Fig. \ref{fig: charts_c_i} (the fixed complex orientation is represented by the arrows). 
Now up to a choice of a complex orientation of the real curve $D_1$, one can glue the charts of the $C_i$'s with those of $D_1$ in a compatible way with respect to the arrows and obtain charts of bidegree $(5,0)$ real curves of type I in $\Sigma_2$ as depicted in Fig. \ref{fig: patchwork_c_i}.
\end{proof}
\begin{prop}
\label{prop: big_first_gluing_2}
The real schemes $1 \quad \sqcup \quad \langle 4 \rangle \quad \sqcup \quad \langle 10 \rangle $, $ \langle 4 \rangle \quad \sqcup \quad  \langle 10 \rangle$, $ \langle 4 \rangle \quad \sqcup \quad \langle 9 \rangle^{\circ}$, $13^{\circ},$ $1 \quad \sqcup \quad \langle 1 \rangle \quad \sqcup \quad \langle 9 \rangle^{\circ}$ and $1 \quad \sqcup \quad \langle 3 \rangle \quad \sqcup \quad \langle 7 \rangle^{\circ}$ are realizable by non-singular real algebraic curves of bidegree $(5,5)$ on the quadric ellipsoid.
\end{prop}
\begin{proof}
Thanks to Remark \ref{rem: fine}, in order to prove the statement we just need to construct bidegree $(5,0)$ real curves of type I (resp. of type II) in $\Sigma_2$ whose charts are as depicted in $a)$ of Fig. \ref{fig: gluing_shustin_orevkov} with
\begin{enumerate}[label=(\arabic*)]
\item $(\alpha, \beta, \gamma)=(5,0,2)$ and $(t,s)=(2,1)$;
\item $(\alpha, \beta, \gamma)=(8,1,0)$ and $(t,s)=(0,3)$;
\item $(\alpha, \beta, \gamma)=(4,1,0)$ and $(t,s)=(3,0)$;
\item $(\alpha, \beta, \gamma)=(4,1,0)$ and $(t,s)=(1,2)$;
\end{enumerate}
and in $b)$ of Fig. \ref{fig: gluing_shustin_orevkov} with $(\alpha, \beta, \gamma)=(5,0,2)$ (resp. in $a)$ of Fig. \ref{fig: gluing_shustin_orevkov} with $(\alpha, \beta, \gamma)=(7,0,1)$ and $(t,s)=(1,2)$), where $\alpha, \beta, \gamma$ and $t,s$ denote numbers of ovals.\\
The strategy is to construct real algebraic curves $C$ of bidegree $(4,0)$ on $\Sigma_2$ (resp. $C'$ of bidegree $(2,4)$ in $\Sigma_1$ intersecting $\mathbb{R}E_1$ in four fixed real points) with chart as depicted in $a)$ of Fig. \ref{fig: charts_shustin_orevkov_bigonal} (resp. $b)$ and $c)$ of Fig. \ref{fig: charts_shustin_orevkov_bigonal}). Thanks to \cite{OreShu16}, there exist real algebraic curves $C$ of bidegree $(4,0)$ and type I, resp. type II, in $\Sigma_2$ whose charts, up to a coordinates change, and arrangements with respect to the coordinates axis $\{x=0\}$ are as depicted in $a)$ of Fig. \ref{fig: charts_shustin_orevkov_bigonal} with $(\alpha,\beta,\gamma)=(5,0,2), (8,1,0)$ and $(4,1,0)$, resp. $(\alpha, \beta, \gamma)=(7,0,1)$. \\ 
\begin{figure}[t!]
\begin{center}
\begin{picture}(100,85)
\put(-175,-20){\includegraphics[width=1.2\textwidth]{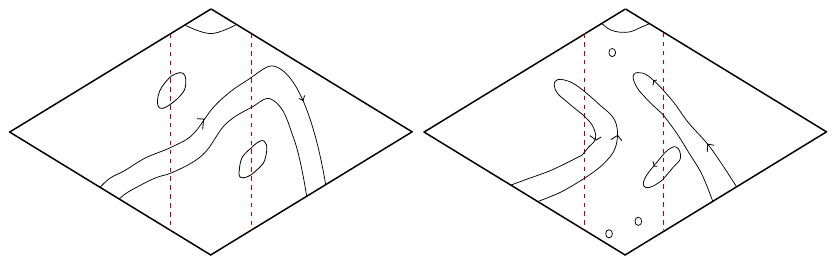}} 
\put(-37,-17){$a)$}
\put (-60,42){$\alpha$}
\put (-41,53){$t$}
\put (-41,23){$s$}
\put (-66,20){$\beta$}
\put (-16,20){$\gamma$}
\put(112,-17){$b)$}
\put (85,42){$\alpha$}
\put (88,20){$\beta$}
\put (132,18){$\gamma$}
\end{picture}
\end{center}
\caption{Charts of real algebraic curves of bidegree $(5,0)$ on $\Sigma_2$.}
\label{fig: gluing_shustin_orevkov}
\end{figure}
\begin{figure}[t!]
\begin{center}
\begin{picture}(100,107)
\put(-175,-17){\includegraphics[width=1.2\textwidth]{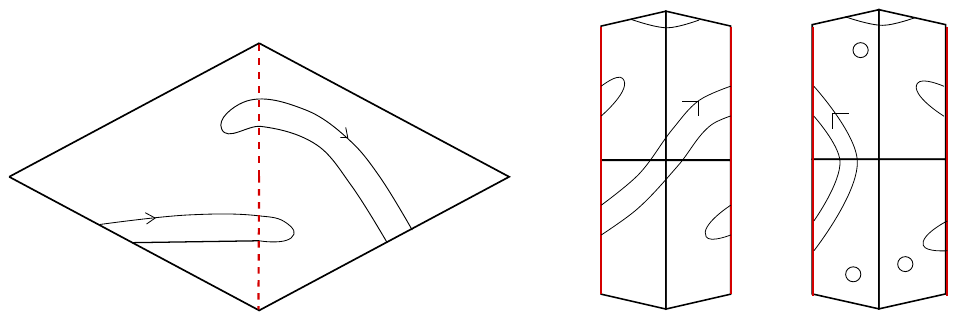}} 
\put(-38,-20){$a)$}
\put (-80,50){$\alpha$}
\put (-70,20){$\beta$}
\put (-10,23){$\gamma$}
\put(108,-20){$b)$}
\put (100,105) {$(0,5)$}
\put (129,99) {$(2,4)$}
\put (100,53){$t$}
\put (100,23){$s$}
\put(186,-20){$c)$}
\put (178,105) {$(0,5)$}
\put (207,99) {$(2,4)$}
\end{picture}
\end{center}
\caption{$a)$ Charts of real algebraic curves $C$ of bidegree $(4,0)$ on $\Sigma_2$. $b)$ - $c)$ Charts of real algebraic curves $C'$ of bidegree $(2,4)$ on $\Sigma_1$.}
\label{fig: charts_shustin_orevkov_bigonal}
\end{figure}
\begin{figure}[h!]
\begin{center}
\begin{picture}(100,60)
\put(-97,-10){\includegraphics[width=1\textwidth]{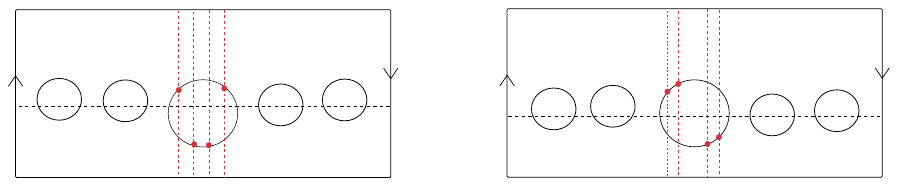}} 
\put(-35,-10){$a)$}
\put (-47,35){\textcolor{red}{$p_1$}}
\put (-23,37){\textcolor{red}{$p_4$}}
\put (-42,12){\textcolor{red}{$p_2$}}
\put (-28,12){\textcolor{red}{$p_3$}}
\put(108,-10){$b)$}
\put (99,36){\textcolor{red}{$p_1$}}
\put (125,16){\textcolor{red}{$p_4$}}
\put (109,40){\textcolor{red}{$p_2$}}
\put (116,12){\textcolor{red}{$p_3$}}
\end{picture}
\end{center}
\caption{$\mathcal{L}$-schemes of real algebraic maximal curves of bidegree $(2,0)$ on $\mathbb{R}\Sigma_5$.}
\label{fig: bigonal_curves}
\end{figure}
\begin{figure}[t!]
\begin{center}
\begin{picture}(100,60)
\put(-137,-20){\includegraphics[width=1\textwidth]{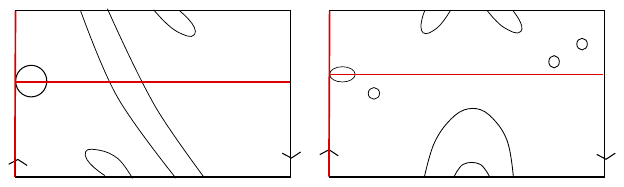}} 
\put(-3,-17){$a)$}
\put(-30,10){$t$}
\put(25,10){$s$}
\put(91,-17){$b)$}
\end{picture}
\end{center}
\caption{$\mathcal{L}$-schemes of real algebraic curves of bidegree $(2,4)$ on $\Sigma_1$.}
\label{fig: shustin_orevkov}
\end{figure}
Let us construct real curves $C'$ in $\Sigma_1$, with the properties described above. Let $\tilde{H}$ be a bidegree $(1,0)$ real curve in $\Sigma_5$ and let us fix ten real points on $\tilde{H}$. As in the proof of Proposition \ref{prop: big_first_gluing}, via small perturbations method it is possible to construct a bidegree $(2,0)$ real curve $\tilde{C}$ intersecting $\tilde{H}$ in the ten fixed points and such that the arrangement of $(\mathbb{R}\Sigma_1, \mathbb{R}\tilde{H}, \mathbb{R}\tilde{C})$ is as depicted in $a)$ of Fig. \ref{fig: bigonal_curves}, resp. $b)$ of Fig. \ref{fig: bigonal_curves}. For any fixed connected component $\mathcal{O}$ of $\mathbb{R}\tilde{C}$ we can pick four real points, $p_1,p_2,p_3,p_4$ on it as depicted in $a)$ of Fig. \ref{fig: bigonal_curves}, resp. $b)$ of Fig. \ref{fig: bigonal_curves}. Then, consider the birational transformation $\beta^{-1}_{p_{1}}\beta^{-1}_{p_{2}}\beta^{-1}_{p_{3}}\beta^{-1}_{p_{4}}: (\Sigma_5,\tilde{C})\rightarrow (\Sigma_1, C')$, as defined in Section \ref{subsec: Hirzebruch_surf}, where we call $p_i$ also the image of $p_i$ via $\beta^{-1}_{p_j}$, $j>i$, $i=1,2,3$. Choose the coordinates axes in $\mathbb{R}\Sigma_1$ such that $\mathbb{R}C'$ has an arrangement as depicted in $b)$ of Fig. \ref{fig: shustin_orevkov}, resp. $a)$ of Fig. \ref{fig: shustin_orevkov} where $t+s=3$. The charts of $C'$ are as depicted in $c)$ of Fig. \ref{fig: charts_shustin_orevkov_bigonal}, resp. $b)$ of Fig. \ref{fig: charts_shustin_orevkov_bigonal}.\\
Applying Viro's patchworking method, gluing the polynomials and the charts of the curves $C$ with those of the curves $C'$ (the patchworking of their charts is depicted in Fig. \ref{fig: gluing_shustin_orevkov}), one constructs the wanted bidegree $(5,0)$ real curves in $\Sigma_2$.
\end{proof}
\begin{prop}
\label{prop: big_first_gluing_3}
The real schemes $1 \quad \sqcup \quad \langle 7 \rangle \quad \sqcup \quad \langle 7 \rangle$, $\langle 7 \rangle \quad \sqcup \quad \langle 7 \rangle$, $\langle 1 \rangle \quad \sqcup \quad \langle 4 \rangle^{\circ}$, $1$ and $0$ are realizable by non-singular real algebraic curves of bidegree $(5,5)$ on the quadric ellipsoid $Q$. 
\end{prop}
\begin{figure}[h!]
\begin{center}
\begin{picture}(100,190)
\put(-137,-10){\includegraphics[width=1\textwidth]{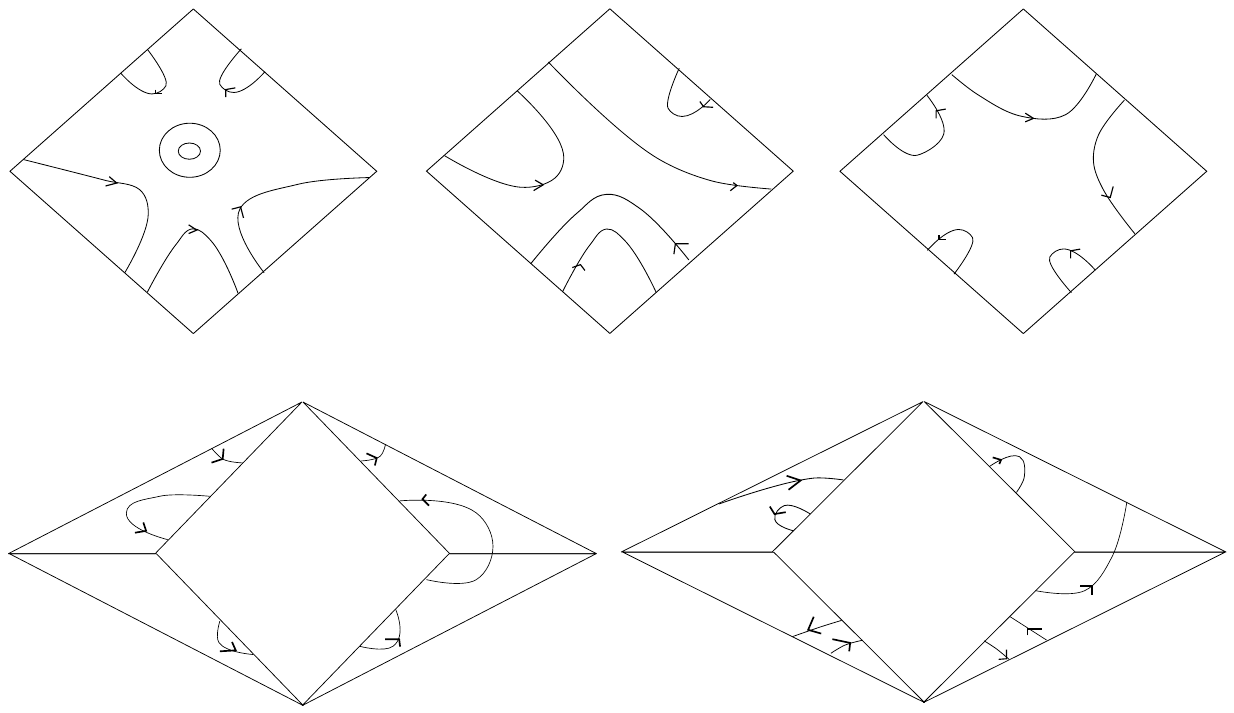}} 
\put(-81,94){$a)R_1$}
\put(-2,94){$b)R_2$ with $(p,q)=(0,6)$}
\put(6,81){$R_3$ with $(p,q)=(6,0)$}
\put(60,125){$p$}
\put (69,161){$q$}
\put(169,94){$c)R_4$}
\put (169,151){$n$}
\put(-48,-17){$d)\tilde{R_4}$}
\put (-85,20){$n$}
\put(140,-17){$e)\tilde{R_3}$}
\put(231,38){$(10,0)$}
\put(157,38){$(5,0)$}
\put(145,82){$(0,5)$}
\put (166,10){$p$}
\end{picture}
\end{center}
\caption{Charts of real algebraic curves.}
\label{fig: quintic_viro_usate}
\end{figure}
\begin{proof}
First of all, in order to realize the real scheme $0$ (resp. $1$), we just perturb the union of five real hyperplane sections in $\mathbb{C}P^3$ not intersecting $\mathbb{R}X$ (resp. whose just one intersects $\mathbb{R}X$) to a smooth surface of degree $5$. For the realization of the real schemes $\langle1\rangle  \quad \sqcup \quad  \langle 4 \rangle^{\circ}$ and $1 \quad \sqcup \quad  \langle 7 \rangle \quad \sqcup \quad  \langle 7 \rangle$ we give some intermediate constructions, then we apply Viro's patchworking method. The strategy is to construct bidegree $(5,0)$ real curves in $\Sigma_2$ whose charts are given by the patchworking of the charts depicted in $a)$ of Fig. \ref{fig: quintic_viro_usate} (resp. $b)$ of Fig. \ref{fig: quintic_viro_usate} with $(p,q)=(0,6)$) with the charts depicted in $d)$ of Fig. \ref{fig: quintic_viro_usate} with $n=4$ (resp. $e)$ of Fig. \ref{fig: quintic_viro_usate} with $p=6$) , where $n$, $p$ and $q$ denote numbers of ovals. Then, Remark \ref{rem: fine} implies that there exist real algebraic curves of bidegree $(5,5)$ in the quadric ellipsoid realizing the real schemes $\langle1\rangle  \quad \sqcup \quad  \langle 4 \rangle^{\circ}$ and $1 \quad \sqcup \quad  \langle 7 \rangle \quad \sqcup \quad  \langle 7 \rangle$.\\
Let us start constructing 
\begin{itemize}
\item real algebraic plane quintics with charts as depicted in $a)$ and $b)$, with $(p,q)=(0,6)$, of Fig. \ref{fig: quintic_viro_usate} and intersecting a given real line in five fixed real points;
\item real algebraic plane quintics with charts as depicted in $b)$, with $(p,q)=(6,0)$, and $c)$, with $n=4$, of Fig. \ref{fig: quintic_viro_usate}.
\end{itemize}
Thanks to \cite[Propositions 4.3B, 4.3C]{Viro89} and \cite{Polo77}, there exist non-singular real affine quintics, with five real asymptotes pointing in five different directions, arranged in the real plane as depicted in Fig. \ref{fig: quintic_viro_piano_affine}, where $(p,q)=(0,6)$, $(6,0)$ and $n=4$. Moreover, due to \cite{Shus83}, for any fixed five directions for the asymptotes, there exist real affine quintics with arrangements in the real plane as depicted in Fig. \ref{fig: quintic_viro_piano_affine}. 
\begin{figure}[h!]
\begin{center}
\begin{picture}(100,80)
\put(-137,-10){\includegraphics[width=1\textwidth]{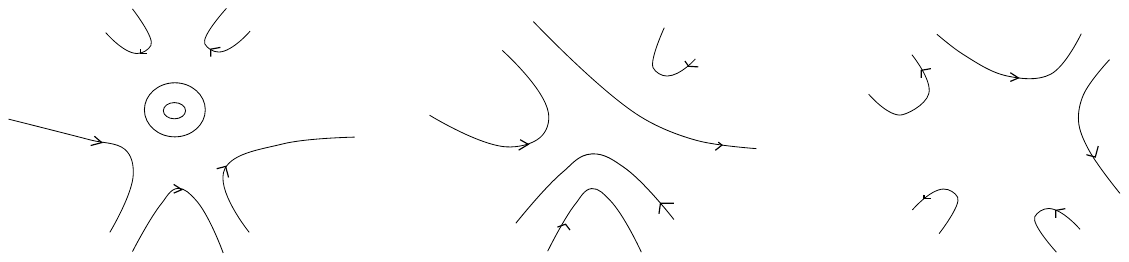}} 
\put(-81,-6){$a)$}
\put(46,-6){$b)$}
\put(60,21){$p$}
\put (69,55){$q$}
\put(169,-6){$c)$}
\put (169,41){$n$}
\end{picture}
\end{center}
\caption{Non-singular real affine quintics.}
\label{fig: quintic_viro_piano_affine}
\end{figure}
It follows that, fixed a real line $L$ in the complex projective plane and fixed five distinct real points on $L$, there exist non-singular real plane quintics $R_1$ and $R_2$ passing through the five given points such that their charts 
are respectively as depicted in $a)$ and $b)$ of Fig. \ref{fig: quintic_viro_usate}, where $(p,q)=(0,6)$. Furthermore, there exist real plane quintics $R_3$ and $R_4$ with charts respectively as depicted in $b)$ and $c)$ of Fig. \ref{fig: quintic_viro_usate}, with $(p,q)=(6,0)$ and $n=4$.\\ 
Now, let us construct real algebraic curves with charts as depicted in $d)$ and $e)$ of Fig. \ref{fig: quintic_viro_usate}, respectively with $n=4$ and  $p=6$. 
Let $$P_h(x,y,z)=\sum\limits_{i+j\leq 5}a_{i,j}x^iy^jz^{5-i-j}$$ be a polynomial of a quintic $R_h$, with $h=3$, $4$. Applying to the polynomials $P_h$ the transformation $T:P_h(x,y,z) \mapsto P_h(xz,yz,x^2)$ 
we construct real algebraic plane curves $\tilde{R}_h$ whose polynomials are $$\tilde{P}_h(x,y,z)=\sum \limits_{i+j \leq 5}a_{i,j}x^{10-i-2j}y^jz^{i+j}$$ and whose charts are as depicted in $d)$ and $e)$ of Fig. \ref{fig: quintic_viro_usate} respectively with $n=4$ and $p=6$. 
In order to realize the real scheme $\langle 1 \rangle \quad \sqcup \quad \langle 4 \rangle^{\circ}$, resp. $1 \quad \sqcup \quad \langle 7 \rangle \quad \sqcup \quad \langle 7 \rangle$, we apply Viro's patchworking method gluing the polynomials and the charts of a real plane quintic $R_1$, resp. $R_2$ with $(p,q)=(0,6)$, with those of $\tilde{R}_4$  with $n=4$, resp. $\tilde{R}_3$ with $p=6$.\\
The construction of a real algebraic curve of bidegree $(5,5)$ on the quadric ellipsoid realizing the real scheme $\langle 7 \rangle \quad \sqcup \quad \langle 7 \rangle$ mimics the construction realizing $1 \quad \sqcup \quad \langle 7 \rangle \quad \sqcup \quad \langle 7 \rangle$, but it requires a variant of the patchworking theorem due to Shustin (\cite[Theorem 5]{Shus05}). The interested reader can find the results of the following paragraph in all their generalities in \cite{Shus05}. Here, we only present an application of \cite[Theorem. 5]{Shus05} in a particular case.\\
Let $\Delta$ be a convex polygon in $\mathbb{R}_+^2$ and let $Tor({\Delta})$ be its associated toric variety. Let $S: \Delta=\bigcup_{i=1}^k\Delta_i$ be a convex subdivision of $\Delta$. Let 
$$f_i= \sum \limits_{(j,h)\in \Delta_i \cap \mathbb{Z}^2}a_{j,h}x^jy^h$$ 
be real polynomials, with $a_{j,h}\in \mathbb{R}$ and such that $C_i= \{f_i=0\}\subset Tor(\Delta_i)$ are non-singular real algebraic curves. Suppose that there exist some faces $\Gamma_{st} \subset \Delta_s \cap \Delta_t$ such that $\Gamma_{st} \not\subset \partial \Delta$ and $z_{st} \in Tor(\Gamma_{st})\cap C_j$ is a real tangency point of $C_j$ with $Tor(\Gamma_{st})$, for $j=s,t$, and locally as depicted in Fig. \ref{fig: tg_point_cha2} on the right. Furthermore, suppose that, out of the tangency points $z_{ts}$, each curve $C_i$ crosses $Tor(\Gamma')$ transversely for any face $\Gamma' \subset \Delta_i$, with $i=1,..,k$.  
Gluing the charts of the $C_i$'s, one does not obtain a chart of a polynomial. Let us consider the following topological construction: replace a neighborhood of the tangency points and of their symmetric points in $\mathbb{R}Tor(\Delta)$ (Fig. \ref{fig: tg_point_cha2}) with a \textit{deformation pattern}, i.e. two disks as depicted in $a)$ or $b)$ of Fig. \ref{fig: deformat_pattern_cha2}. Then, Shustin (\cite{Shus05}) proved that such topological construction is realizable algebraically and it produces a chart of a non-singular real polynomial in $Tor(\Delta)$.
\begin{figure}[!h]
\begin{picture}(100,65)
\put(33,0){\includegraphics[width=0.80\textwidth]{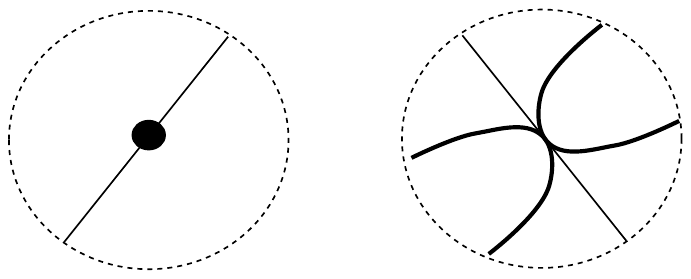}} 
\end{picture}
\caption{ }
\label{fig: tg_point_cha2}
\end{figure}
\begin{figure}[!h]
\begin{picture}(100,55)
\put(77,-12){$a)$}
\put(290,-12){$b)$}
\put(0,0){\includegraphics[width=1.00\textwidth]{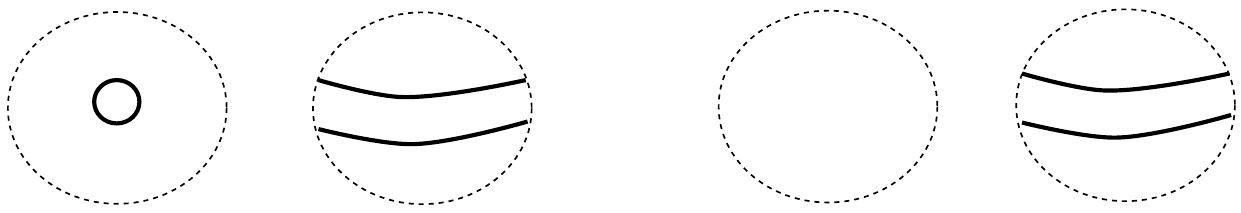}} 
\end{picture}
\caption{ }
\label{fig: deformat_pattern_cha2}
\end{figure}
We want to construct a bidegree $(5,0)$ real algebraic curve in $\Sigma_2$ with chart as depicted in $b)$ of Fig. \ref{fig: last}. \\
First of all, let us construct two non-singular plane quintics $C_1$, $C_2$ such that
\begin{itemize}
\item $C_1$ passes through three fixed real points on a given real line $L$, it is tangent to $L$ at another fixed real point and it has chart as depicted in $a)$ of Fig. \ref{fig: quintic_viro_usate_tg}, with $(p,q)=(0,6)$; 
\item $C_2$ has chart as depicted in $a)$ of Fig. \ref{fig: quintic_viro_usate_tg}, with $(p,q)=(6,0)$.
\end{itemize}
\begin{figure}[h!]
\begin{center}
\begin{picture}(100,210)
\put(-137,-19){\includegraphics[width=1\textwidth]{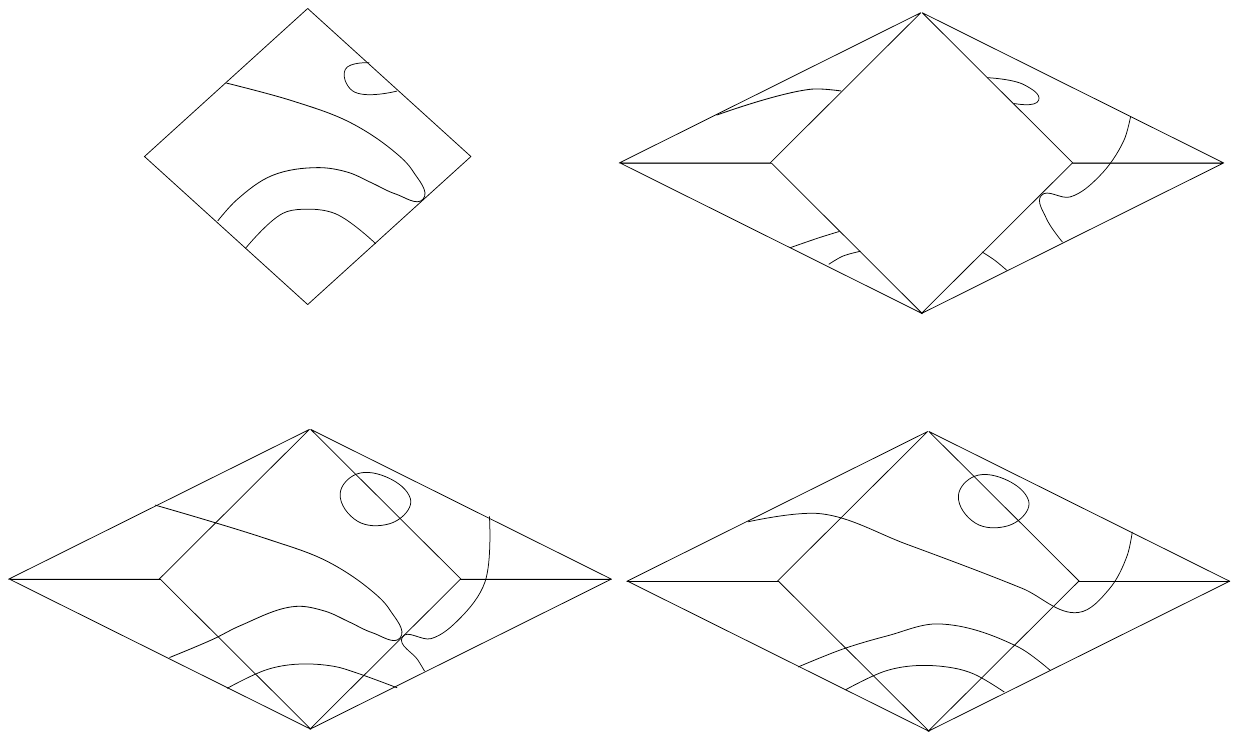}} 
\put(-128,94){$a)$ Chart of a real algebraic plane quintic.}
\put (-29,183){$q$}
\put (-22,153){$p$}
\put(64,94){$b)$ Chart of a real algebraic plane curve.}
\put (167,146){$p$}
\put(231,168){$(10,0)$}
\put(157,168){$(5,0)$}
\put(144,214){$(0,5)$}
\put(-48,-9){$c)$}
\put (-18,45){$q$}
\put (-19,19){$p$}
\put(64,-9){$d)$ Chart of a non-singular real algebraic}
\put(64,-20){curve of bidegree $(5,0)$ on $\Sigma_2$.}
\put (163,53){$q$}
\put (164,19){$p$}
\end{picture}
\end{center}
\caption{ }
\label{fig: quintic_viro_usate_tg}
\end{figure}
\begin{figure}[h!]
\begin{center}
\begin{picture}(100,120)
\put(-127,-5){\includegraphics[width=1\textwidth]{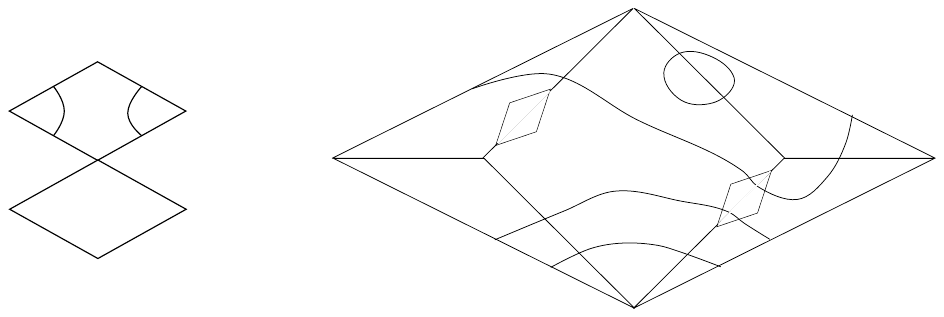}} 
\put(-100,-9){$a)$ Deformation pattern.}
\put(-50,93){$(0,1)$}
\put(-10,70){$(2,0)$}
\put(-30,57){$(0,-1)$}
\put(54,-4){$b)$ Chart of a non-singular real algebraic}
\put(54,-15){curve of bidegree $(5,0)$ on $\Sigma_2$.}
\put (148,64){$q$}
\put (143,32){$p$}
\end{picture}
\end{center}
\caption{ }
\label{fig: last}
\end{figure}
In order to construct $C_1$, let us fix four distinct real points $p_1,p_2,p_3,p_4$ on a given real line $L$. In some local chart one has $L=\{(x,y) \in \mathbb{R}^2:y=0\}$ and $p_i=(x_i,0)$, with $i=1,2,3,4$. 
Fixed any small real number $\varepsilon \not=0$, there exists a non-singular real plane quintic $C_0$ passing through $p_1$, $p_2$, $p_3$, $ (x_4+\varepsilon,0)$, $(x_4-\varepsilon,0)$ and whose chart is as depicted in $b)$ of Fig. \ref{fig: quintic_viro_usate}, where $(p,q)=(0,6)$. The quintic $C_0$ is locally given as
$$\{(x,y) \in \mathbb{R}^2:\sum\limits_{i+j \leq 5, j\not = 0}a_{i,j}x^iy^j +(x-x_1)(x-x_2)(x-x_3)(x-x_4+\varepsilon)(x-x_4-\varepsilon)=0\}.$$
Let us consider the real one-parameter family $\mathcal{H}_t$, with $t \in [0,1]$, of non-singular real plane quintics $$C_t:=\{\sum\limits_{i+j \leq 5, j\not = 0}a_{i,j}x^iy^j +(x-x_1)(x-x_2)(x-x_3)(x-x_4+(1-t)\varepsilon)(x-x_4-(1-t)\varepsilon)=0\}.$$ The real quintic $C_1$ of the family is the wanted real curve: it passes through $p_1,p_2,p_3$, it is tangent to $L$ at $p_4$ and it has chart as depicted in $a)$ of Fig. \ref{fig: quintic_viro_usate_tg}, with $(p,q)=(0,6)$.\\  
Since there exists a real plane quintic with chart as depicted in $b)$ of Fig. \ref{fig: quintic_viro_usate} with $(p,q)=(6,0)$, it is easy to see that there exists a real algebraic plane quintic $C_2$ whose chart is as depicted in $a)$ of Fig. \ref{fig: quintic_viro_usate_tg}, with $(p,q)=(6,0)$. Applying the transformation $T$ to the real quintic $C_2$, one constructs a non-singular real plane curve $\tilde{C}_2$ with chart as depicted in $b)$ of Fig. \ref{fig: quintic_viro_usate_tg}, where $p=6$.\\
Unlike in the construction of a bidegree $(5,5)$ real curve in the quadric ellipsoid realizing the real scheme $1 \quad \sqcup \quad \langle 7 \rangle \quad \sqcup \quad \langle 7 \rangle$, one can not glue the charts of $C_1$ and $\tilde{C}_2$ as depicted in $c)$ of Fig. \ref{fig: quintic_viro_usate_tg} (where $p$ and $q$ are both equal to $6$), because we do not obtain a chart of a polynomial. But, the variant of the patchworking method developed by Shustin allows us to replace a neighborhood of the tangency point and its symmetric point with a \textit{deformation pattern}; see $a)$ and $b)$ of Fig. \ref{fig: last}. At the end, we obtain a chart of a real algebraic curve of bidegree $(5,0)$ in $\Sigma_2$ as depicted in $d)$ of Fig. \ref{fig: quintic_viro_usate_tg}, with $p=6$ and $q=6$.
\end{proof}
\bibliographystyle{alpha}
\bibliography{biblio.bib}
Matilde Manzaroli, University of Oslo, UiO\\
Postboks 1053 Blindern, 0316 OSLO, Norway\\
E-mail adress: manzarom@math.uio.no
\end{document}